\documentclass[final,3p,times]{elsarticle}

\usepackage{amssymb}
\usepackage{amsthm}

\usepackage{amsmath}           
\usepackage{etoolbox}          
\usepackage{bookmark}
\usepackage{nicefrac}

\hypersetup{
	colorlinks=true,
	linkcolor=blue,
	filecolor=magenta,
	citecolor=teal,      
	urlcolor=blue,
}

\journal{CAMWA}

\newtheorem{theorem}{Theorem}[section]
\newtheorem{lemma}[theorem]{Lemma}
\newtheorem{assumption}[theorem]{Assumption}

\theoremstyle{definition}
\newtheorem{definition}[theorem]{Definition}

\newtheorem{corollary}[theorem]{Corollary}

\theoremstyle{remark}
\newtheorem{remark}[theorem]{Remark}

\newcommand{\reftext}[1]{{#1}}

\usepackage{amssymb,amsmath,amsfonts,latexsym}
\usepackage{pstricks,pst-plot,pst-node,pst-text,pst-3d}
\usepackage{enumerate}
\usepackage[english]{babel}
\usepackage{geometry,booktabs, fancyhdr}

\usepackage{amsmath,latexsym,mathrsfs,textcomp, fontenc,yhmath}
\usepackage{amssymb,amsfonts,amscd,amssymb,amsthm}
\usepackage{graphics,verbatim,url}
\usepackage{epic, curves}
\newcommand{\doi}[1]{\url{http://dx.doi.org/#1}}
\usepackage{index}
\usepackage[T1]{fontenc}
\usepackage[utf8]{inputenc}
\usepackage{tabularx,ragged2e,booktabs,caption,subcaption}
\usepackage{pgf,tikz}
\usepackage{mathrsfs,mathtools}
\usetikzlibrary{arrows}
\usepackage{pgflibraryarrows}
\usepackage{pgflibrarysnakes}
\usetikzlibrary{calc}
\usepackage[notcite, notref]{showkeys}

\begin{document}
	\begin{frontmatter}
		
		\title{Analytic and Gevrey class regularity\\ for parametric semilinear reaction-diffusion problems\\ and 
			applications in uncertainty quantification}
		
		
		\author[UOL]{Alexey Chernov}
		\ead{alexey.chernov@uni-oldenburg.de}
		\author[UOL]{T{\`u}ng L{\^e}}
		\ead{tung.le@uni-oldenburg.de}
\begin{abstract}
	We investigate a class of parametric elliptic semilinear partial differential
	equations of second order with homogeneous essential boundary conditions,
	where the coefficients and the right-hand side (and hence the solution)
	may depend on a parameter. This model can be seen as a reaction-diffusion
	problem with a polynomial nonlinearity in the reaction term. The efficiency
	of various numerical approximations across the entire parameter space is
	closely related to the regularity of the solution with respect to the parameter.
	We show that if the coefficients and the right-hand side are analytic or
	Gevrey class regular with respect to the parameter, the same type of parametric
	regularity is valid for the solution. The key ingredient of the proof is
	the combination of the alternative-to-factorial technique from our previous
	work \citep{ChernovLe2023} with a novel argument for the treatment of the
	power-type nonlinearity in the reaction term. As an application of this
	abstract result, we obtain rigorous convergence estimates for numerical
	integration of semilinear reaction-diffusion problems with random coefficients
	using Gaussian and Quasi-Monte Carlo quadrature. Our theoretical findings
	are confirmed in numerical experiments.
\end{abstract}
%
%
%

\end{frontmatter}
%
\section{Introduction}
\label{sec:_intro_VTEX1}

Elliptic semilinear problems arise in numerous applications in natural
sciences and engineering. Prominent examples are reaction-diffusion-type
problems with nonlinear reaction (reproduction) terms for modelling of
various processes such as phase separation, combustion, soil-moisture-physics,
biological population genetics, etc. For analysis of parametric semilinear
problems we refer to works by Hansen and Schwab \cite{HansenSchwab13},
where the particular case of a stochastic parameter perturbation has been
addressed, see also the recent work \cite{bahn2023semilinear} for semilinear
eigenvalue problems under uncertainty. The regularity of the solution of
the problem with respect to the parameter is important for construction
of efficient numerical approximations in the parameter domain. For example,
if the quantity of interest is solution's average value over a prescribed
parameter domain, Monte Carlo and Quasi-Monte Carlo methods can be applied
for the numerical integration. However, the use of Quasi-Monte Carlo integration
does only pay off if the solution features certain higher regularity properties.
In this paper, we investigate the parametric regularity
of the solution to a semilinear problem, focusing on two norms:
$H^{1}_{0}(D)$ and $H^{2}(D)$. Subsequently, we illustrate their implications
in Gauss-Legendre quadrature and random shift lattice rules through numerical
experiments.

Let us consider a prototypic real second-order elliptic semilinear partial
differential equation of the general form
%
\begin{align}
\label{gen_equation_VTEX1}
-C_{m}^{2}\nabla \cdot (\boldsymbol{a}(\boldsymbol{x},
\boldsymbol{y})\nabla u(\boldsymbol{x},\boldsymbol{y})) +b(
\boldsymbol{x},\boldsymbol{y}) \left [u(\boldsymbol{x},\boldsymbol{y})
\right ]^{m} &= C_{m}\,f(\boldsymbol{x},\boldsymbol{y}),
&
(\boldsymbol{x},\boldsymbol{y})\in D \times U,
\nonumber\\
u(\boldsymbol{x},\boldsymbol{y})&=0,
&
(\boldsymbol{x},\boldsymbol{y}) \in \partial D \times U,
%
\end{align}
where the derivative operator $\nabla $ acts in the physical variable
$\boldsymbol{x}\in D$, and $D$ is a bounded Lipschitz domain in
$\mathbb R^{d}$. The above semilinear problem could be reduced to linear
by choosing $b\equiv 0$ or $m= 1$.
For the parametric regularity of linear problems, refer
to works such as
\cite{CohenDevoreSchwab2010,CohenDevoreSchwab2011,KuoSchwabSloan2012}.
The vector of parameters
$\boldsymbol{y}= (y_{1},y_{2},\dots ) \in U$ has either finitely many or
countably many components. For example, if $\boldsymbol{y}$ is a random
parameter, the model with
$U:= [-\frac{1}{2},\frac{1}{2}]^{\mathbb N}$ and
$\boldsymbol{y}\in U$ being a countably-dimensional vector of independently
and identically distributed uniform random variables has been frequently
used in the literature
\cite{CohenDevoreSchwab2010,CohenDevoreSchwab2011,KuoSchwabSloan2013,KuoNuyens2016}.
The necessary restrictions on the power $m\in \mathbb{N}$ and the dimension
$d$ of the domain $D$ fall within the range specified in \reftext{\eqref{cond:_m_d_restriction_VTEX1}} throughout this paper. We denote the corresponding
set of parameters $(d,m)$ by $\mathcal{M}$, as outlined in
\cite{HansenSchwab13}
%
\begin{equation}
\label{cond:_m_d_restriction_VTEX1}
\begin{split} d=1 \text{ or } d=2:& \quad m\in \mathbb N,
	\\
	d=3: H^{1}_{0}(D) \hookrightarrow L^{6}(D)& \quad \text{hence } 1\leq m
	\leq 5,
	\\
	d=4: H^{1}_{0}(D) \hookrightarrow L^{4}(D)& \quad \text{hence } 1\leq m
	\leq 3,
	\\
	d=5: H^{1}_{0}(D) \hookrightarrow L^{10/3}(D)& \quad \text{hence } 1
	\leq m\leq 2,
	\\
	d=6: H^{1}_{0}(D) \hookrightarrow L^{3}(D)& \quad \text{hence } 1\leq m
	\leq 2,
	\\
	d\geq 7:& \quad m=1,
\end{split}
%
\end{equation}
and $C_{m}$ is the constant of Sobolev embedding
$H^{1}_{0}(D)\hookrightarrow L^{m+1}(D)$, see e.g.
\cite{Mizuguchi2017}. Clearly, the rescaling of \reftext{\eqref{gen_equation_VTEX1}} with
$C_{m}$ is made for the sake of convenience and has no effect in the parametric
regularity of the solution. Without loss of generality, in the following
we assume that the coefficients
$a(\cdot ,\boldsymbol{y}),b(\cdot ,\boldsymbol{y})$ and
$\left \|{f(\cdot ,\boldsymbol{y})}\right \|_{H^{-1}(D)}$ admit the uniform
bounds
%
\begin{equation}
\label{ab-bounds}
1 \leq a(\boldsymbol{x},\boldsymbol{y})\leq \frac{\overline{a}}{2} ,
\qquad \left |b(\boldsymbol{x},\boldsymbol{y})\right | \leq
\frac{\overline{b}}{2}, \qquad \left \|{f(\cdot ,\boldsymbol{y})}
\right \|_{H^{-1}(D)} \leq \frac{\overline{f}}{2}
\end{equation}
for all $\boldsymbol{y}\in U$ and almost all $\boldsymbol{x}\in D$. Thus,
for every fixed $\boldsymbol{y}\in U$ and under additional assumptions,
the problem \reftext{\eqref{gen_equation_VTEX1}} is well-posed, see for example~\cite{HansenSchwab13}. Since the coefficients depend on the parameters
$\boldsymbol{y}$, the solution $u$ will depend on $\boldsymbol{y}$ as well.
Particularly, if $\boldsymbol{y}$ is random, then
$u(\boldsymbol{x},\boldsymbol{y})$ will be random too.

In this paper we present a rigorous regularity analysis for the solution
$u$ with respect to the parameter $\boldsymbol{y}$ in the general case
where the given data $a, b$ and $f$ are infinitely differentiable functions
of $\boldsymbol{y}$ belonging to the Gevrey class $G^{\delta}$ for some
fixed $\delta \geq 1$. The scale of Gevrey classes is a nested scale of
the parameter $\delta $ that fills the gap between analytic and
$C^{\infty}$ functions
%
\begin{equation}
\label{Gdelta-nested}
{\mathcal A}= G^{1} \subset \, G^{\delta} \subset G^{
	\delta '} \subset C^{\infty}, \qquad
1 < \delta < \delta '.
\end{equation}
The case of analytic functions (i.e.
${\mathcal A}=G^{1}$, Gevrey-$\delta $-class function
with $\delta = 1$) is the simplest and arguably most important of this
scale and has been addressed for parametric/stochastic semilinear problem
before, see \cite{HansenSchwab13}. The above mentioned work considers the
coefficient represented by an affine parametrization and proves the analyticity
of the solution using elegant complex analysis arguments.

Besides the complex variable argument, the real-variable argument is an alternative 
powerful tool to prove the analytic regularity of the parametric solution. For the
analysis of the eigenvalue problems we refer to \cite{Gilbert2019}. However,
the direct application of the real-variable argument typically leads to
suboptimal estimates. To overcome this, in \cite{ChernovLe2023}, we suggest
a modified argument, namely \emph{alternative-to-factorial technique}, and
obtain optimal regularity for the eigenvalue problem. The aim of this paper
is to extend this approach to the case of parametric
semilinear problems.

The structure of the paper is as follows. In Section~\ref{sec:_Remedy_VTEX1} we introduce the falling factor notation, which is the
main tool for our \emph{alternative-to-factorial technique}, originally introduced
in \cite{ChernovLe2023}. In Section~\ref{sec:_Gevrey_VTEX1} we introduce Gevrey-classes 
and formulate the regularity assumptions on the coefficients of
the semilinear problem \reftext{\eqref{gen_equation_VTEX1}}. In Section~\ref{sec:EVPtheory} we summarize the properties of elliptic semilinear
problems needed for the forthcoming regularity analysis. In Section~\ref{sec:_main_result_VTEX1}
we present the proof of the main results, namely the Gevrey-$\delta $ class
regularity for the solution with respect to the $H^{1}_{0}(D)$-seminorm
in \reftext{Theorem~\ref{gevrey_regularity_for_semilinear_VTEX1}} and the higher spatial
regularity results in \reftext{Theorem~\ref{thm:2}}. The meaning and validity of
the main regularity results is illustrated by the applications and numerical
experiments in Section~\ref{sec:_num_exp_VTEX1}.

\section{Preliminaries}
\label{sec:_Preliminary_VTEX1}

\subsection{The falling factorial estimates}
\label{sec:_Remedy_VTEX1}

The deficiency of the real-variable inductive argument for nonlinear problems
is a consequence of the Leibniz product rule and the triangle inequality.
It can be seen already in one-dimensional case,
see \cite[Section 2.1]{ChernovLe2023} and
\cite[Chapter 1]{Krantz2002}. To overcome these difficulties, we utilize the \emph{alternative-to-factorial technique} as introduced in
\cite[Section 2.2]{ChernovLe2023}. To summarize this we collect some elementary
results on the falling factorial as follows.

For a given $q \in \mathbb R$ and a non-negative integer
$n \in \mathbb N_{0}$ the \emph{falling factorial} is defined as
%
\begin{equation}
\label{def-fallfact}
(q)_{n} := \left \{
\begin{array}{c@{\quad}l}
	1, & n=0,
	\\
	q(q-1)\dots (q-n+1), & n \geq 1.
\end{array}
\right .
\end{equation}
For $q<1$ the falling factorial $(q)_{n}$ is a sign-alternating sequence
of $n$. To further simplify the notation and avoid keeping track of the
sign alteration, we denote the
\emph{absolute value of the falling factorial of $\frac{1}{2}$} by
\begin{equation*}
\left [\tfrac{1}{2} \right ]_{n} := \left |\left (\tfrac{1}{2}
\right )_{n}\right |.
\end{equation*}
This notation appears somewhat non-standard, but quite convenient, as we
will see in the forthcoming analysis. The two sided-estimate
%
\begin{align}
\label{ff-estimates}
\left [\tfrac{1}{2} \right ]_{n} \leq n! \leq 2 \cdot 2^{n} \left [
\tfrac{1}{2} \right ]_{n} ,
\end{align}
is rather crude but sufficient for our analysis, see
\cite[Section 2.1]{ChernovLe2023} for a refined version.

The following combinatorial identities are remarkable properties of the
falling factorial. The first and the second estimates in \reftext{\eqref{ff-main-id-2}} are stated here for a shifted summation range, cf.
\cite[Lemma 2.3]{ChernovLe2023}, and, thus, require new proofs given below.
%
\begin{lemma}
\label{lem:ff-main-id-2}
For all integers $n\geq 1$ and $k\geq 2$ the following identities hold
%
\begin{equation}
	\label{ff-main-id-2}
	\sum _{i=1}^{n} \binom{n}{i} \left [\tfrac{1}{2} \right ]_{i} \,
	\left [\tfrac{1}{2} \right ]_{n+1-i} = \left [\tfrac{1}{2} \right ]_{n+1}, 
	\qquad \sum _{i=1}^{k-1} \binom{k}{i} \left [\tfrac{1}{2} \right ]_{i}
	\, \left [\tfrac{1}{2} \right ]_{k-i} = 2 \left [\tfrac{1}{2} \right ]_{k},
	\qquad \sum _{i=1}^{k} \binom{k}{i} \left [\tfrac{1}{2} \right ]_{i}
	\, \left [\tfrac{1}{2} \right ]_{k-i} = 3 \left [\tfrac{1}{2} \right ]_{k}
	.
\end{equation}
\end{lemma}
\begin{proof}
We choose the function $f(x)=\frac{1}{2}(1-\sqrt{1-x})$ and
%
\begin{align}
	g(x)=f(x)f'(x) =\left (\frac{1}{2}(1-\sqrt{1-x}\,)\right )\cdot
	\left (\frac{1}{4}\frac{1}{\sqrt{1-x}}\right ) =\frac{1}{8}\left (
	\frac{1}{\sqrt{1-x}}-1\right )=\frac{1}{2}f'(x)-\frac{1}{8}.
	\label{eq8}
\end{align}
From \cite[Section 2.2]{ChernovLe2023}, we know that
$f^{(n)}(0)=\frac{1}{2} \left [\tfrac{1}{2} \right ]_{n} $ for all
$n \in \mathbb N$. Thus, on the one hand, for all $n\geq 1$ we have
\begin{align*}
	g^{(n)}(0)= \frac{1}{2}f^{(n+1)}(0) = \frac{1}{4}\left [\tfrac{1}{2}
	\right ]_{n+1} .
\end{align*}
On the other hand, by Leibniz product rule and since $f(0)=0$, we have
\begin{align*}
	g^{(n)}(0)= \sum _{i=1}^{n} \binom{n}{i} f^{(i)}(0) f^{(n+1-i)}(0)
	= \frac{1}{4} \sum _{i=1}^{n} \binom{n}{i} \left [
	\tfrac{1}{2} \right ]_{i} \left [\tfrac{1}{2} \right ]_{n+1-i} .
\end{align*}
This shows the first identity in \reftext{\eqref{ff-main-id-2}}. The second identity
follows for $f(x)=\frac{1}{2}(1-\sqrt{1-x})$ and $g=f^{2}$, see e.g.
\cite[Lemma 2.3]{ChernovLe2023}. Increasing both sides of this identity
by $\left [\tfrac{1}{2} \right ]_{k} $, we observe that the third identity
in \reftext{\eqref{ff-main-id-2}} is also valid.
\end{proof}
%
\begin{corollary}
\label{cor2.2}
With the convention that the empty sum equals zero, the \reftext{Lemma~\ref{lem:ff-main-id-2}} extends to all non-negative integers
$n \in \mathbb N_{0}$ as
%
\begin{align}
	\label{ff-bound-1}
	\sum _{i=1}^{n} \binom{n}{i} \left [\tfrac{1}{2} \right ]_{i} \,
	\left [\tfrac{1}{2} \right ]_{n+1-i} &\leq \left [\tfrac{1}{2}
	\right ]_{n+1} ,\\
	\sum _{i=1}^{n-1}
	\binom{k}{i} \left [\tfrac{1}{2} \right ]_{i} \, \left [\tfrac{1}{2}
	\right ]_{n-i} &\leq
	2 \left [\tfrac{1}{2} \right ]_{n},
	\label{ff-bound-2}\\
	\sum _{i=1}^{n} \binom{k}{i} \left [\tfrac{1}{2}
	\right ]_{i} \, \left [\tfrac{1}{2} \right ]_{n-i} &\leq
	3 \left [\tfrac{1}{2} \right ]_{n}.
	\label{ff-bound-3}
\end{align}
\end{corollary}
%

\subsection{Multi-index notation}
\label{sec:_Multiindex_VTEX1}

The following standard multi-index notations will be used
in what follows, see e.g.
\cite{BoutetDeMonvel1967,CohenDevoreSchwab2010}. We denote the countable
set of finitely supported sequences of nonnegative integers by
%
\begin{equation}
\label{cF-def}
\mathcal F:= \left \{ \boldsymbol{\nu }=(\nu _{1},\nu _{2},\dots )~:~
\nu _{j}\in \mathbb N_{0}, \text{ and } \nu _{j} \neq 0
\text{ for only a finite number of } j \right \} \subset \mathbb N^{
	\mathbb N},
\end{equation}
where the summation $\boldsymbol{\alpha }+ \boldsymbol{\beta }$ and the
partial order relations $\boldsymbol{\alpha }< \boldsymbol{\beta }$ and
$\boldsymbol{\alpha }\leq \boldsymbol{\beta }$ of elements in
$\boldsymbol{\alpha }, \boldsymbol{\beta }\in \mathcal F$ are understood
componentwise. We write
\begin{align*}
\left |\boldsymbol{\nu }\right | := \sum _{j\geq 1} \nu _{j}, \qquad
\qquad \boldsymbol{\nu }! := \prod _{j\geq 1} \nu_{j}!, \qquad \qquad
\boldsymbol{R}^{\boldsymbol{\nu }}= \prod _{j\geq 1} R_{j}^{\nu_{j}}
\end{align*}
for the absolute value, the multifactorial and the power with the
multi-index $\boldsymbol{\nu }$ and a sequence
$\boldsymbol{R}=\{R_{j}\}_{j\geq 1}$ of positive real numbers. Notice that
$\left |\boldsymbol{\nu }\right |$ is finite if and only if
$\boldsymbol{\nu }\in \mathcal F$. For
$\boldsymbol{\nu }\in \mathcal F$ supported in
$\left \{1,2,\dots , n\right \}$, we define the partial derivative with
respect to the variables $\boldsymbol{y}$
\begin{align*}
\partial ^{\boldsymbol{\nu }} u =
\frac{\partial ^{\left |\boldsymbol{\nu }\right |}u}
{\partial y_{1}^{\nu _{1}} \partial y_{2}^{\nu _{2}} \dots \partial y_{n}^{\nu _{n}}}.
\end{align*}
For two multi-indicies
$\boldsymbol{\nu }, \boldsymbol{\eta }\in \mathcal F$ we define the binomial
coefficient by
\begin{equation*}
\binom{\boldsymbol{\nu }}{\boldsymbol{\eta }} = \prod _{j
	\geq 1} \binom{\nu _{j}}{\eta _{j}}.
	\end{equation*}
	
	The above multi-index notations are handy for treatment
	of multiparametric objects. The following technical Lemma is instrumental
	for the forthcoming analysis.
	%
	\begin{lemma}
\label{lem2.3}
For two multi-indicies
$\boldsymbol{\nu }, \boldsymbol{\eta }\in \mathcal F$ satisfying
$\boldsymbol{\eta }\leq \boldsymbol{\nu }$, a unit
multi-index $\boldsymbol{e}$ and $\delta \geq 1$ we have
%
\begin{equation}
	\label{multiindex-est-1}
	(|\boldsymbol{\nu }-\boldsymbol{\eta }|!)^{\delta -1} (|
	\boldsymbol{\eta }|!)^{\delta -1} \leq (|\boldsymbol{\nu }|!)^{
		\delta -1},
\end{equation}
%
\begin{equation}
	\sum _{\boldsymbol{0}<\boldsymbol{\eta }< \boldsymbol{\nu }}
	\binom{\boldsymbol{\nu }}{\boldsymbol{\eta }} \left [
	\tfrac{1}{2} \right ]_{|\boldsymbol{\eta }|} \left [\tfrac{1}{2}
	\right ]_{|\boldsymbol{\nu }-\boldsymbol{\eta }|} \leq
	2 \left [\tfrac{1}{2} \right ]_{|\boldsymbol{\nu }|},
	\label{multiindex-est-7}
\end{equation}
%
\begin{equation}
	\sum _{\boldsymbol{0}<\boldsymbol{\eta }\leq \boldsymbol{\nu }}
	\binom{\boldsymbol{\nu }}{\boldsymbol{\eta }} \left [
	\tfrac{1}{2} \right ]_{|\boldsymbol{\eta }|} \left [\tfrac{1}{2}
	\right ]_{|\boldsymbol{\nu }-\boldsymbol{\eta }|} \leq 3 \left[
	\tfrac{1}{2} \right ]_{|\boldsymbol{\nu }|},
	\label{multiindex-est-3}
\end{equation}
%
\begin{equation}
	\sum _{
		\boldsymbol{0}< \boldsymbol{\eta }\leq
		\boldsymbol{\nu }} \binom{\boldsymbol{\nu }}{\boldsymbol{\eta }}
	\left [\tfrac{1}{2} \right ]_{\left |\boldsymbol{\nu }+
		\boldsymbol{e}-\boldsymbol{\eta }\right |} \left [\tfrac{1}{2}
	\right ]_{\left |\boldsymbol{\eta }\right |} \leq
	\left [\tfrac{1}{2} \right ]_{\left |
		\boldsymbol{\nu }+\boldsymbol{e}\right |} ,
	\label{multiindex-est-6}
\end{equation}
%
\begin{equation}
	\sum _{\boldsymbol{0}< \boldsymbol{\eta }\leq \boldsymbol{\nu }}
	\sum _{
		\boldsymbol{0}< \boldsymbol{\ell }\leq
		\boldsymbol{\eta }} \binom{\boldsymbol{\nu }}{\boldsymbol{\eta }}
	\binom{\boldsymbol{\eta }}{\boldsymbol{\ell }}
	\left [\tfrac{1}{2} \right ]_{\left |\boldsymbol{\eta }-
		\boldsymbol{\ell }\right |} \left [\tfrac{1}{2} \right ]_{\left |
		\boldsymbol{\ell }\right |} \left [\tfrac{1}{2} \right ]_{\left |
		\boldsymbol{\nu }+\boldsymbol{e}-\boldsymbol{\eta }\right |} \leq
	3\left [\tfrac{1}{2} \right ]_{\left |
		\boldsymbol{\nu }+\boldsymbol{e}\right |} .
	\label{multiindex-est-5}
\end{equation}
\end{lemma}
\begin{proof}
Notice that for two non-negative integers $n!\cdot m! \leq (n+m)!$ and
therefore
\begin{align*}
	|\boldsymbol{\nu }-\boldsymbol{\eta }|! |\boldsymbol{\eta }|! \leq (|
	\boldsymbol{\nu }-\boldsymbol{\eta }|+ |\boldsymbol{\eta }|)! = |
	\boldsymbol{\nu }|!.
\end{align*}
Since $(\cdot )^{\delta -1}$ is an increasing function for
$\delta \geq 1$, the estimate \reftext{\eqref{multiindex-est-1}} follows. According
to \cite[Lemma 7.1]{ChernovLe2023}, we have
%
\begin{align}
	\label{iden:_Chu-Vandermonde_VTEX1}
	\sum _{{\left |\boldsymbol{\eta }\right |=r}\atop{\boldsymbol{\eta } \leq
			\boldsymbol{\nu }}} \binom{\boldsymbol{\nu }}{\eta}=
	\binom{|\boldsymbol{\nu }|}{r},
\end{align}
which is sometimes called generalized Vandermonde or Chu-Vandermonde identity.
This together with \reftext{\eqref{ff-bound-3}} imply the estimate
\begin{align*}
	\sum _{\boldsymbol{0}<\boldsymbol{\eta }< \boldsymbol{\nu }}
	\binom{\boldsymbol{\nu }}{\boldsymbol{\eta }} \left [
	\tfrac{1}{2} \right ]_{|\boldsymbol{\nu }-\boldsymbol{\eta }|} \left [
	\tfrac{1}{2} \right ]_{|\boldsymbol{\eta }|} = \sum _{
		r=1}^{
		\left |\boldsymbol{\nu }\right |-1} \sum _{{\left |
			\boldsymbol{\eta }\right |=r}\atop{\boldsymbol{\eta }\leq
			\boldsymbol{\nu }}} \binom{\boldsymbol{\nu }}{\boldsymbol{\eta }}
	\left [\tfrac{1}{2} \right ]_{|\boldsymbol{\nu }|-r} \left [
	\tfrac{1}{2} \right ]_{r} = \sum _{r=1}^{
		\left |\boldsymbol{\nu }\right |-1} \binom{|
		\boldsymbol{\nu }|}{r} \left [\tfrac{1}{2} \right ]_{|
		\boldsymbol{\nu }|-r} \left [\tfrac{1}{2} \right ]_{r} \leq
	2\left [\tfrac{1}{2} \right ]_{|\boldsymbol{\nu }|}
	.
\end{align*}
This shows inequality \reftext{\eqref{multiindex-est-3}}. Similarly, we derive bounds \reftext{\eqref{multiindex-est-3}} and \reftext{\eqref{multiindex-est-6}} by applying \reftext{\eqref{iden:_Chu-Vandermonde_VTEX1}} to \reftext{\eqref{ff-bound-3}} and \reftext{\eqref{ff-bound-1}}, respectively. The final estimate \reftext{\eqref{multiindex-est-5}} follows by the consecutive application of \reftext{\eqref{multiindex-est-3}} and \reftext{\eqref{multiindex-est-6}}.
\end{proof}

\subsection{Gevrey-class and analytic functions}
\label{sec:_Gevrey_VTEX1}

The following definition of Gevrey-$\delta $ functions with countably many
parameters will be used in our regularity analysis in Section~\ref{sec:_main_result_VTEX1}.

\begin{definition}
\label{def:G-delta-def}
Let $\delta \geq 1$, $B$ be a Banach space,
$I \subset \mathbb R^{\mathbb N}$ be an open domain and a function
$f:I \to B$ be such that its $\boldsymbol{y}$-derivatives
$\partial ^{\boldsymbol{\nu }} f:I \to B$ are continuous for all
$\boldsymbol{\nu }\in \mathcal F$. We say the function
$f$ is of class Gevrey-$\delta $ if for each $y_{0} \in I$ there exist
an open set $J\subseteq I$, and strictly positive constants
$\boldsymbol{R}= (R_{1},R_{2},\dots) \in \mathbb R_{>0}^{\mathbb N}$ and $C \in \mathbb R_{>0}$ that the derivatives of $f$ satisfy
the bounds
%
\begin{equation}
	\label{G-delta-def}
	\|\partial ^{\boldsymbol{\nu }} f(\boldsymbol{y})\|_{B} \leq
	\frac{C}{\boldsymbol{R}^{\boldsymbol{\nu }}} (|\boldsymbol{\nu }|!)^{
		\delta}, \qquad \forall \boldsymbol{y}\in J, \quad \forall
	\boldsymbol{\nu }\in \mathcal F.
\end{equation}
In this case we write $f \in G^{\delta}(U,B)$.
\end{definition}

\reftext{Definition~\ref{def:G-delta-def}} is also suitable for the case of finitely
many parameters $\boldsymbol{y}$. In particular, when
$\boldsymbol{y}= (y_{1},\dots ,y_{s})$, $B = \mathbb R$ or
$\mathbb C$ and $\delta =1$, the bound \reftext{\eqref{G-delta-def}} guarantees convergence
of the power series of $f$ and therefore characterizes the class of analytic
functions of $s$ variables, see e.g. \cite[Section~2.2]{Krantz2002} and
\cite[Remark~2.6]{ChernovLe2023}. This property follows from the bound
$|\boldsymbol{\nu }|! \leq s^{|\boldsymbol{\nu }|} \boldsymbol{\nu }!$
that is valid for a multi-index $\boldsymbol{\nu }$ with
$s$ nonzero components. Notice that otherwise estimate \reftext{\eqref{G-delta-def}} does not guarantee convergence of the power series
of $f$. Moreover, the scale $G^{\delta}$ grows
monotonically with $\delta $ in the sense of \reftext{\eqref{Gdelta-nested}}.

We now make an assumption on the coefficients, which, in particular, ensure
that the solution of the semilinear problem \reftext{\eqref{gen_equation_VTEX1}} is Gevrey-class
regular.
%
\begin{assumption}
\label{Assumption}
For all fixed values
$\boldsymbol{y}\in U \subset  \mathbb R^{s}$ with
$s \in \mathbb N$, the coefficients
$a(\boldsymbol{y}),b(\boldsymbol{y}) \in L^{\infty}(D)$ and
$f(y)\in V^{*}$. The functions $a$, $b$ are of Gevrey class
$G^{\delta}(U,L^{\infty}(D))$ and $f$ is of Gevrey class
$G^{\delta}(U,V^{*})$, i.e. for all
$\boldsymbol{\nu }\in \mathbb N^{s}$ there exist $\boldsymbol{R}$ independent
of $s$ such that
\begin{align*}
	\left \|{\partial ^{\boldsymbol{\nu }} a(\boldsymbol{y})}\right \|_{L^{
			\infty }(D)} \leq \frac{\overline{a}}{2}
	\frac{(\left |\boldsymbol{\nu }\right |!)^{\delta }}{(2\boldsymbol{R})^{\boldsymbol{\nu }}}
	, \qquad \left \|{\partial ^{\boldsymbol{\nu }} b(\boldsymbol{y})}
	\right \|_{L^{\infty }(D)} \leq \frac{\overline{b}}{2}
	\frac{(\left |\boldsymbol{\nu }\right |!)^{\delta }}{(2\boldsymbol{R})^{\boldsymbol{\nu }}}
	, \qquad \left \|{\partial ^{\boldsymbol{\nu }} f(\boldsymbol{y})}
	\right \|_{V^{*}} \leq \frac{\overline{f}}{2}
	\frac{(\left |\boldsymbol{\nu }\right |!)^{\delta }}{(2\boldsymbol{R})^{\boldsymbol{\nu }}}
	.
\end{align*}
\end{assumption}
Notice that for $\boldsymbol{\nu }= \boldsymbol{0}$ \reftext{Assumption~\ref{Assumption}} agrees with the upper bounds in \reftext{\eqref{ab-bounds}}. Notice
also that the components of $\boldsymbol{R}$ are readily scaled by the
factor of $2$. This leads to no loss of generality, but helps to shorten
the forthcoming expressions. For example, in view of \reftext{\eqref{ff-estimates}} \reftext{Assumption~\ref{Assumption}} immediately implies
%
\begin{align}
\label{abf_assumption_VTEX1}
\left \|{\partial ^{\boldsymbol{\nu }} a(\boldsymbol{y})}\right \|_{L^{
		\infty }(D)} \leq
\frac{\overline{a} \left [\tfrac{1}{2} \right ]_{\left |\boldsymbol{\nu }\right |} }{\boldsymbol{R}^{\boldsymbol{\nu }}}
(\left |\boldsymbol{\nu }\right |!)^{\delta -1} , \qquad \left \|{
	\partial ^{\boldsymbol{\nu }} b(\boldsymbol{y})}\right \|_{L^{\infty }(D)}
\leq
\frac{\overline{b} \left [\tfrac{1}{2} \right ]_{\left |\boldsymbol{\nu }\right |} }{\boldsymbol{R}^{\boldsymbol{\nu }}}
(\left |\boldsymbol{\nu }\right |!)^{\delta -1} , \qquad \left \|{
	\partial ^{\boldsymbol{\nu }} f(\boldsymbol{y})}\right \|_{V^{*}}
\leq
\frac{\overline{f} \left [\tfrac{1}{2} \right ]_{\left |\boldsymbol{\nu }\right |} }{\boldsymbol{R}^{\boldsymbol{\nu }}}(
\left |\boldsymbol{\nu }\right |!)^{\delta -1} .
\end{align}

The definition of the norms used above is standard
and will be recalled in the beginning of the next section.

\section{Elliptic semilinear PDEs with countably many parameters}
\label{sec:EVPtheory}

For a fixed $\boldsymbol{y}\in U$ the variational formulation of \reftext{\eqref{gen_equation_VTEX1}} reads
%
\begin{equation}
\label{variational_form_VTEX1}
C_{m}^{2}\int _{D} a(\boldsymbol{x},\boldsymbol{y})\nabla u(
\boldsymbol{x},\boldsymbol{y})\cdot \nabla v(\boldsymbol{x}) + \int _{D}
b(\boldsymbol{x},\boldsymbol{y})[ u(\boldsymbol{x},\boldsymbol{y}) ]^{m}
\, v(\boldsymbol{x}) = C_{m}\int _{D} f(\boldsymbol{x},\boldsymbol{y})
\, v(\boldsymbol{x}).
\end{equation}
The H\"older inequality implies that the second integral is well-defined
for $u(\cdot , \boldsymbol{y}), v(\cdot )\in L^{m+1}(D)$. By the Sobolev
embedding theorem this is guaranteed for $H^{1}_{0}(D)$ functions under
restrictions on the range of $m$ as readily announced in \reftext{\eqref{cond:_m_d_restriction_VTEX1}}. We now collect the required notations and
facts from the theory of variational semilinear problems. By
$L^{p}(D)$ and $L^{\infty}(D)$ we denote the spaces of $p$-power integrable
and bounded functions equipped with standard norms.

Throughout the paper, when it is unambiguous we will drop the
$\boldsymbol{x}$-dependence when referring to a function defined on
$D$ at a parameter value~$\boldsymbol{y}$. We introduce the Sobolev spaces
$V:=H^{1}_{0}(D)$, its dual $V^{*}:=H^{-1}(D)$ equipped the following norms
\begin{align*}
\left \|{u}\right \|_{V}:= C_{m}\left \|{u}\right \|_{H^{1}_{0}(D)},
\qquad \left \|{f}\right \|_{V^{*}}:=\left \|{f}\right \|_{H^{-1}(D)} =
\sup _{{v\in V}\atop{v\neq 0}}
\frac{\int _{D} f\, v}{\left \|{v}\right \|_{H^{1}_{0}(D)}} = \sup _{{v
		\in V}\atop{v\neq 0}}
\frac{\left \langle{f},{v} \right \rangle }{\left \|{v}\right \|_{V}},
\end{align*}
where the duality pairing on $V\times V^{*}$ is denoted by
$\left \langle{\cdot },{\cdot } \right \rangle $ as
%
\begin{align}
\left \langle{g},{v} \right \rangle := C_{m} \int _{D} g\,v, \qquad
\forall g\in V^{*} \text{ and } \forall v\in V.
\label{eq22}
\end{align}
Moreover, we have
%
\begin{align}
\left \langle{f},{v} \right \rangle \leq C_{m}\left \|{f}\right \|_{H^{-1}(D)}
\left \|{v}\right \|_{H^{1}_{0}(D)} = \left \|{f}\right \|_{V^{*}}
\left \|{v}\right \|_{V}.
\label{Sobolev_dual_ineq_VTEX1}
\end{align}

\subsection{$L^{p}$ spaces and the H\"older inequality}
\label{sec3.1}

The H\"older inequality will be frequently used in various ways in the
following. As we will see, the H\"older exponents will
often be rational. In order to keep the technical presentation
clear, we propose the following simplifying notation for the classical Lebesgue
spaces and their norms
%
\begin{equation}
\label{LL-def}
\mathcal L(p/k) := L^{\frac{p}{k}}(D), \qquad \|f\|_{\mathcal L(p/k)} :=
\left |\int _{D} |f|^{\frac{p}{k}} \, dx \right |^{\frac{k}{p}},
\qquad 0 < k \leq p < \infty .
\end{equation}
Although this notation makes sense for general positive real $p$ and
$k$, in this paper they will only take integer values. Since
$k \leq p$, the spaces $\mathcal L(p/k)$ are Banach spaces and the triangle
inequality (also known as Minkowski inequality in this
special case) is valid
\begin{equation*}
\|f + g\|_{\mathcal L(p/k)} \leq \|f\|_{\mathcal L(p/k)}+ \|g\|_{
	\mathcal L(p/k)}, \qquad 0 < k \leq p < \infty .
	\end{equation*}
	The slash symbol in the notation $\mathcal L(p/k)$ (we intentionally don't
	use the comma to avoid possible confusion with Lorentz spaces) is used
	as a splitter between two parameters $p$ and $k$ and simultaneously indicates
	the fraction $\frac{p}{k}$. The following Lemma states the generalized
	H\"older inequalities in the notation of the Lebesgue spaces
	$\mathcal L(p/k)$, where the convenience of the two-parameter notation
	becomes apparent. This calculus plays an important role in the proof of
	the main regularity results in Section~\ref{sec:_main_result_VTEX1}.
	
	\begin{lemma}%
\label{lem:_Hoelder-type_ineq_VTEX1}
Let $k,n,p$ be positive integers satisfying $p \geq k+n $. Then for any
$f \in \mathcal L(p/k)$, $f \in \mathcal L(p/n)$ it holds that
%
\begin{align}
	\label{adjusted_Hoelder_VTEX1}
	\|fg\|_{\mathcal L(p/k+n)} \leq \|f\|_{\mathcal L(p/k)} \|g\|_{
		\mathcal L(p/n)}.
\end{align}
Moreover, for all positive integers $\ell \leq p$ and
$f_{1}, \dots , f_{\ell}\in \mathcal L(p/1)$ their product belongs to
$\mathcal L(p/\ell )$ and there holds
%
\begin{equation}
	\label{multiple_Hoelder_VTEX1}
	\big\|\prod _{i=1}^{\ell }f_{i}\big\|_{\mathcal L(p/\ell )} \leq
	\prod _{i=1}^{\ell }\big\|f_{i}\big\|_{\mathcal L(p/1)}.
\end{equation}
\end{lemma}
\begin{proof}
The classical H\"older inequality implies for H\"older
conjugates $q = \tfrac{k+n}{k}$ and $q' = \tfrac{k+n}{n}$
\begin{align*}
	\int _{D} |fg|^{\frac{p}{k+n}} &\leq \left ( \int _{D} |f|^{
		\frac{pq}{k+n}} \right )^{\frac{1}{q}} \left ( \int _{D} |g|^{
		\frac{pq'}{k+n}} \right )^{\frac{1}{q'}} = \left ( \int _{D} |f|^{
		\frac{p}{k}} \right )^{\frac{k}{k+n}} \left ( \int _{D} |g|^{
		\frac{p}{n}} \right )^{\frac{n}{k+n}}.
\end{align*}
Raising both sides to the power $\tfrac{k+n}{p}$ directly implies the estimate \reftext{\eqref{adjusted_Hoelder_VTEX1}}. The second inequality \reftext{\eqref{multiple_Hoelder_VTEX1}} follows from \reftext{\eqref{adjusted_Hoelder_VTEX1}} by induction,
since for any integer $n$ satisfying $n+1\leq p$ it holds that
\begin{equation*}
	\big\|\prod _{i=1}^{n+1} f_{i}\big\|_{\mathcal L(p/n+1)} \leq \big\|
	\prod _{i=1}^{n} f_{i}\big\|_{\mathcal L(p/n)} \big\|f_{n+1}\big\|_{
		\mathcal L(p/1)} \leq \prod _{i=1}^{n+1} \big\|f_{i}\big\|_{
		\mathcal L(p/1)},
\end{equation*}
where the inductive assumption has been used in the last step. This finishes
the proof.
\end{proof}
%
\begin{remark}
\label{rem3.2}
Notice that the statement of the above lemma is valid for a much larger
range of parameters. So \reftext{\eqref{adjusted_Hoelder_VTEX1}} is valid for all real
$k,n,p \in (0,\infty ]$ without further conditions. Notice that
$L^{\alpha}(D)$ becomes a quasi-Banach space for $0 < \alpha < 1$, where
the triangle inequality in no longer valid. This also implies a nonsymmetric
version of \reftext{\eqref{multiple_Hoelder_VTEX1}} with general
$f_{i} \in L^{p_{i}}(D)$, see e.g. \cite[Section 8]{WhZy2015}. In what
follows we only require the reduced parameter range as in \reftext{Lemma~\ref{lem:_Hoelder-type_ineq_VTEX1}} and particularly estimates \reftext{\eqref{adjusted_Hoelder_VTEX1}} and \reftext{\eqref{multiple_Hoelder_VTEX1}} in the specific notation \reftext{\eqref{LL-def}}.
\end{remark}

Rational order Lebesgue spaces are convenient to claim Gevrey-$\delta $ regularity of integer powers $f^{k}$ of a function
$f = f(\boldsymbol{x},\boldsymbol{y})$ if $f$ is Gevrey-$\delta $ regular
with respect to $\boldsymbol{y}$.
%
\begin{lemma}%
\label{upper_bound_for_f_derivative_VTEX1}
Let $p \geq 1$,
$\boldsymbol{\nu }\in
\mathcal F\setminus \left \{\boldsymbol{0}\right
\}$,
$\boldsymbol{R}= (R_{1},R_{2},\dots ) \in \mathbb R_{>0}^{\mathbb N}$ be strictly positive constants,
$\rho \in \mathbb R_{\geq 2}$ and
$C_{f} \in \mathbb R_{>0}$ so that the $\boldsymbol{y}$-derivatives of
a sufficiently regular
$f = f(\boldsymbol{x},\boldsymbol{y})\in {\mathcal B}(0,C_{f})$ for
$(\boldsymbol{x},\boldsymbol{y}) \in D \times J$ satisfy the bounds
%
\begin{equation}
	\label{f-delta-def}
	\|\partial ^{\boldsymbol{\eta }} f(\boldsymbol{y})\|_{L^{p}(D)} \leq
	\frac{C_{f} \left [\tfrac{1}{2} \right ]_{|\boldsymbol{\eta }|} }{\rho \,\boldsymbol{R}^{\boldsymbol{\eta }}}
	(|\boldsymbol{\eta }|!)^{\delta -1}, \qquad \forall \boldsymbol{y}
	\in J, \quad \forall \boldsymbol{\eta }\leq \boldsymbol{\nu }
	\text{ and } \left |\boldsymbol{\eta }\right |\geq 1.
\end{equation}
Then, for all integer powers $1 \leq k\leq p$ the function $f^{k}$ and
all its $\boldsymbol{y}$-derivatives belong to $L^{\frac{p}{k}}(D)$ and the following
estimates are valid
%
\begin{equation}
	\label{fk_derivative_VTEX1}
	\left \|{\partial ^{\boldsymbol{\nu }}(f(\boldsymbol{y})^{k})}\right
	\|_{\mathcal L(p/k)} \leq
	\frac{3^{k-1}C_{f}^{k} \left [\tfrac{1}{2} \right ]_{\left |\boldsymbol{\nu }\right |} }{\rho \,\boldsymbol{R}^{{\boldsymbol{\nu }}}}
	(\left |\boldsymbol{\nu }\right |!)^{\delta -1}, \qquad \forall
	\boldsymbol{y}\in J.
\end{equation}
\end{lemma}
\begin{proof}
We prove this statement by induction with respect to $k$. The basis of
the induction for $k=1$ is identical to the assumption \reftext{\eqref{f-delta-def}}. Suppose now that the statement \reftext{\eqref{fk_derivative_VTEX1}} is valid for $k = n$ and show it for
$k=n+1 \leq p$. By the Leibniz product rule we have
\begin{equation*}
	\partial ^{\boldsymbol{\nu }} (f^{n+1}) = \sum _{\boldsymbol{0}\leq
		\boldsymbol{\eta }\leq \boldsymbol{\nu }}
	\binom{\boldsymbol{\nu }}{\boldsymbol{\eta }} \partial ^{
		\boldsymbol{\eta }}(f^{n})\partial ^{\boldsymbol{\nu }-
		\boldsymbol{\eta }}f.
\end{equation*}
Since $p \geq n+1$, the triangle inequality is valid, hence
\begin{equation*}
	\|\partial ^{\boldsymbol{\nu }} (f^{n+1})\|_{\mathcal L(p/n+1)} \leq
	\|f^{n}\|_{\mathcal L(p/n)} \|\partial ^{
		\boldsymbol{\nu }}f\|_{\mathcal L(p/1)} + \|\partial ^{
		\boldsymbol{\nu }}(f^{n})\|_{\mathcal L(p/n)} \|f\|_{\mathcal L(p/1)} +
	\sum _{\boldsymbol{0}< \boldsymbol{\eta }< \boldsymbol{\nu }}
	\binom{\boldsymbol{\nu }}{\boldsymbol{\eta }} \|\partial ^{
		\boldsymbol{\eta }}(f^{n})\|_{\mathcal L(p/n)} \|\partial ^{
		\boldsymbol{\nu }-\boldsymbol{\eta }}f\|_{\mathcal L(p/1)},
\end{equation*}
where the H\"older inequality \reftext{\eqref{adjusted_Hoelder_VTEX1}} has been used in
the last step. According to the inductive assumption, the
first summand is bounded as in equation \reftext{\eqref{f-delta-def}} with
$\boldsymbol{\eta }=\boldsymbol{\nu }$, while for the second summand, the
bound \reftext{\eqref{fk_derivative_VTEX1}} with $k=n$ holds true. This together with the
bounds \reftext{\eqref{multiindex-est-1}}, \reftext{\eqref{multiindex-est-7}} and $\rho \geq 2$, implies
\begin{equation*}
	\begin{split} \|\partial ^{\boldsymbol{\nu }} (f^{n+1})\|_{\mathcal L(p/n+1)}
		&
		\leq
		\frac{C_{f}^{n+1} \left [\tfrac{1}{2} \right ]_{|\boldsymbol{\nu }|} }{\rho \,\boldsymbol{R}^{\boldsymbol{\nu }}}
		(|\boldsymbol{\nu }|!)^{\delta -1} +
		\frac{3^{n-1} C_{f}^{n+1} \left [\tfrac{1}{2} \right ]_{|\boldsymbol{\nu }|} }{\rho \, \boldsymbol{R}^{\boldsymbol{\nu }}}
		(|\boldsymbol{\nu }|!)^{\delta -1} \\
		& \qquad 
		+ \sum _{\boldsymbol{0}< \boldsymbol{\eta }<
			\boldsymbol{\nu }} \binom{\boldsymbol{\nu }}{\boldsymbol{\eta }}
		\frac{3^{n-1}C_{f}^{n} \left [\tfrac{1}{2} \right ]_{|\boldsymbol{\eta }|} }{\rho \,\boldsymbol{R}^{\boldsymbol{\eta }}}
		(|\boldsymbol{\eta }|!)^{\delta -1}
		\frac{C_{f}\left [\tfrac{1}{2} \right ]_{|\boldsymbol{\nu }-\boldsymbol{\eta }|} }{\rho \,\boldsymbol{R}^{\boldsymbol{\nu }-\boldsymbol{\eta }}}
		(|\boldsymbol{\nu }-\boldsymbol{\eta }|!)^{\delta -1}
		\\
		&
		\leq
		\frac{3^{n-1} C_{f}^{n+1} }{\rho \, \boldsymbol{R}^{\boldsymbol{\nu }}}
		(|\boldsymbol{\nu }|!)^{\delta -1} \left (
		\frac{ \left [\tfrac{1}{2} \right ]_{|\boldsymbol{\nu }|} }{3^{n-1}} +
		\left [\tfrac{1}{2} \right ]_{|\boldsymbol{\nu }|} + \frac{1}{\rho }
		\sum _{\boldsymbol{0}< \boldsymbol{\eta }< \boldsymbol{\nu }}
		\binom{\boldsymbol{\nu }}{\boldsymbol{\eta }} \left [\tfrac{1}{2}
		\right ]_{|\boldsymbol{\eta }|} \left [\tfrac{1}{2} \right ]_{|
			\boldsymbol{\nu }-\boldsymbol{\eta }|} \right )
		\\
		&
		\leq
		\frac{3^{n-1} C_{f}^{n+1} \left [\tfrac{1}{2} \right ]_{|\boldsymbol{\nu }|} }{\rho \, \boldsymbol{R}^{\boldsymbol{\nu }}}
		(|\boldsymbol{\nu }|!)^{\delta -1} \left (\frac{1}{3^{n-1}}+1+
		\frac{2}{\rho }\right ) \leq
		\frac{3^{n} C_{f}^{n+1} \left [\tfrac{1}{2} \right ]_{|\boldsymbol{\nu }|} }{\rho \, \boldsymbol{R}^{\boldsymbol{\nu }}}
		(|\boldsymbol{\nu }|!)^{\delta -1} .
	\end{split}
\end{equation*}
This finishes the proof.
\end{proof}

From the Sobolev embedding theorem, for every $v\in V$ and
$f\in V^{*}$ we have
%
\begin{align}
\left \|{u}\right \|_{\mathcal L(m+1/1)}= \left \|{u}\right \|_{L^{m+1}(D)}
\leq C_{m} \left \|{u}\right \|_{H^{1}_{0}(D)} = \left \|{u}\right \|_{V}
\label{Sobolev_ineq_VTEX1}
,
\end{align}
where the Sobolev embedding constant $C_{m}$ could be calculated explicitly
as in \cite{Mizuguchi2017}.

\subsection{Well-posedness of the variational formulation}
\label{sec3.2}

For a fixed $\boldsymbol{y}$ we define the bilinear form
$A_{\boldsymbol{y}}:V\times V\rightarrow \mathbb R$ and a nonlinear operator
$T_{\boldsymbol{y}}(w,v): V \times V \rightarrow \mathbb R$
%
\begin{align}
\label{AT-def}
A_{\boldsymbol{y}}(w,v) := C_{m}^{2}\int _{D} a(\boldsymbol{y})
\nabla u\cdot \nabla v, \quad T_{\boldsymbol{y}}(w,v) :=\int _{D} b (
\boldsymbol{y})w^{m} v.
\end{align}
In the view of bounds \reftext{\eqref{ab-bounds}} and \reftext{Lemma~\ref{lem:_Hoelder-type_ineq_VTEX1}}, we have
%
\begin{align}
A_{\boldsymbol{y}} (w,w) \geq \| w\|_{V}^{2}, \qquad A_{
	\boldsymbol{y}} (w,v) &\leq \frac{\overline{a}}{2} \| w\|_{V} \| v\|_{V},
\qquad w,v \in V,
\label{A_bound_VTEX1}
\\
T_{ \boldsymbol{y}} ( w,v) &\leq \frac{\overline{b}}{2} \| w\|_{V}^{m}
\| v\|_{V}, \qquad w,v \in V.
\label{T_bound_VTEX1}
\end{align}
Thus, for every $\boldsymbol{y}\in U$, the variational equivalent of \reftext{\eqref{gen_equation_VTEX1}} is the problem of finding a solution $u\in V$ such
that
%
\begin{align}
\label{semi_prob_VTEX1}
A_{\boldsymbol{y}}(u(\boldsymbol{y}),v)+T_{\boldsymbol{y}}(u(
\boldsymbol{y}),v)=\left \langle{f(\boldsymbol{y})},{v} \right
\rangle \quad \forall v\in V.
\end{align}
Note that the uniqueness of a solution to \reftext{\eqref{semi_prob_VTEX1}} is generally
not guaranteed. For instance, the (real-valued) problem for
$a\equiv b\equiv 1$ with certain restriction on $m$ even has infinitely
many solutions with arbitrarily large norms, see
\cite[Theorem 7.2 and Remark 7.3]{Struwe2008} and
\cite{HansenSchwab13}.

In the following we introduce two different assumptions (\reftext{Assumption~\ref{General_b_Assump_VTEX1}} and \reftext{Assumption~\ref{Possitive_b_Assump_VTEX1}}) that are
sufficient to guarantee the existence and uniqueness of the solution to \reftext{\eqref{variational_form_VTEX1}}. Roughly speaking, \reftext{Assumption~\ref{General_b_Assump_VTEX1}} admits indefinite reaction term
$T_{\boldsymbol{y}}$, but requires that $\overline{b}$ and
$\overline{f}$ cannot be large simultaneously. If, however,
$T_{\boldsymbol{y}}$ is nonnegative, no further restrictions are required,
see \reftext{Assumption~\ref{Possitive_b_Assump_VTEX1}}. In both cases the unique solution
is bounded
%
\begin{equation}
\label{u-bound}
\|u\|_{V} \leq \overline{u},
\end{equation}
where the upper bound $\overline{u}$ will be determined below.

We now give the details of the argument. The following assumption naturally
extends the result in \cite{HansenSchwab13}.
%
\begin{assumption}
\label{General_b_Assump_VTEX1}
For a fixed integer $m\geq 1$, there exists a positive constant
$\gamma < 1$ such that $\overline{b}$ and $\overline{f}$ satisfy
\begin{align*}
	\frac{\overline{b}}{2} = \frac{\gamma}{m\, \overline{f}^{m-1}} .
\end{align*}
\end{assumption}
For the case $m=1$, the problem in \reftext{\eqref{variational_form_VTEX1}} turns into
a linear reaction-diffusion problem. The bilinear form for this problem
is $V$-coercive for $\frac{\overline b}{2} < 1$
%
\begin{equation}
\label{A1m1}
C_{m}^{2}\int _{D} a(\boldsymbol{y}) |\nabla w|^{2} + \int _{D} b (
\boldsymbol{y})w^{2} \geq \|w\|_{V}^{2} - \frac{\overline b}{2} \|w\|_{L^{2}(D)}^{2}
\geq \bigg(1- \frac{\overline b}{2}\bigg) \|w\|_{V}^{2}.
\end{equation}
This property is guaranteed by \reftext{Assumption~\ref{General_b_Assump_VTEX1}} for
$m = 1$.

We now show that \reftext{Assumption~\ref{General_b_Assump_VTEX1}} {also guarantees that} \reftext{\eqref{variational_form_VTEX1}} has a unique solution by means of the Banach fixed-point
theorem. Let $u_{0}(\boldsymbol{y})=0$ and define
$u_{n+1}(\boldsymbol{y})$ as the unique solution of
%
\begin{align}
\label{contracting_mapping_VTEX1}
A_{\boldsymbol{y}}(u_{n+1}(\boldsymbol{y}),v) = \left \langle{f(
	\boldsymbol{y})},{v} \right \rangle - T_{\boldsymbol{y}}(u_{n}(
\boldsymbol{y}),v) \qquad \forall v\in V.
\end{align}
The following Lemma shows that the sequence $\{u_{n}\}$ never leaves the
closed set
\begin{equation*}
{\mathcal B}(0,\overline{f}):=\left \{v\in V: \left \|{v}\right \|_{V}
\leq{\overline{f}}\right \}
\end{equation*}
and converges to a limit in ${\mathcal B}(0,\overline{f})$. Obviously, when
the sequence $\left \{u_{n}(\boldsymbol{y})\right \}$ admits a limit point,
it would be a solution of \reftext{\eqref{variational_form_VTEX1}}. Indeed, the following
Lemma proves the above statement.
%
\begin{lemma}%
\label{lem:A1m2}
For every $\boldsymbol{y}\in U$ and an integer $m\geq 2$,
the sequence $\left \{\|u_{n}(\boldsymbol{y})\|_{V}\right \}$ is bounded
by $\overline{f}$ and the sequence $\left \{u_{n}(\boldsymbol{y})\right \}$ converges to a fixed point in
${\mathcal B}(0,{\overline{f}})$.
\end{lemma}
\begin{proof}
We will prove boundedness of the sequence by induction with respect to
$n$. The Lax-Milgram Lemma and \reftext{\eqref{ab-bounds}} imply
\begin{equation*}
	\|u_{1}\|_{V} \leq \|f\|_{V^{*}} \leq \frac{\overline f}{2}
\end{equation*}
and hence $u_{1}\in {\mathcal B}(0,{\overline{f}})$. Assume now that all
$u_{1}, \dots , u_{n}$ belong to this neighbourhood and prove that the
same holds for $u_{n+1}$. For this, we substitute
$v=u_{n+1}(\boldsymbol{y})$ into \reftext{\eqref{contracting_mapping_VTEX1}} and recall \reftext{\eqref{Sobolev_dual_ineq_VTEX1}}, \reftext{\eqref{ab-bounds}}, \reftext{\eqref{T_bound_VTEX1}} and \reftext{Assumption~\ref{General_b_Assump_VTEX1}} to obtain
\begin{align*}
	\left \|{u_{n+1}(\boldsymbol{y})}\right \|_{V} \leq
	\frac{\overline{f}}{2} + \frac{\overline{b}}{2} \left \|{u_{n}(
		\boldsymbol{y})}\right \|_{V}^{m} \leq \frac{\overline{f}}{2} +
	\frac{\gamma}{m\, \overline{f}^{m-1}}\, \overline{f}^{m} \leq
	\overline{f},
\end{align*}
where in the last step we have used that $\gamma <1$ and $m\geq 2$. This
shows that $u_{n}\in {\mathcal B}(0,{\overline{f}})$ for all $n$. We now
prove that the sequence converges to a limit in $V$. We have that
$u_{n}(\boldsymbol{y})$ is the solution of
\begin{align*}
	A_{\boldsymbol{y}}(u_{n}(\boldsymbol{y}),v) = \left \langle{f(
		\boldsymbol{y})},{v} \right \rangle - T_{\boldsymbol{y}}(u_{n-1}(
	\boldsymbol{y}),v) \quad \forall v\in V.
\end{align*}
Subtract both sides of the above equation from \reftext{\eqref{contracting_mapping_VTEX1}} and set
$v=u_{n+1}(\boldsymbol{y})-u_{n}(\boldsymbol{y})$ to obtain
\begin{align*}
	C_{m}^{2}\int _{D} a(\boldsymbol{y}) \left |\nabla u_{n+1}(
	\boldsymbol{y})-\nabla u_{n}(\boldsymbol{y})\right |^{2} = \int _{D} b(
	\boldsymbol{y})\left (\left [u_{n-1}(\boldsymbol{y})\right ]^{m}-
	\left [u_{n}(\boldsymbol{y})\right ]^{m}\right )(u_{n+1}(
	\boldsymbol{y})-u_{n}(\boldsymbol{y})).
\end{align*}
The left-hand side is bounded by $\|u_{n+1} - u_{n}\|_{V}^{2}$ from below.
To obtain an upper bound for the right-hand side, recall the elementary
identity $a^{m}-b^{m}=(a-b)\sum _{j=0}^{m-1}a^{m-1-j}b^{j}$ and \reftext{Lemma~\ref{lem:_Hoelder-type_ineq_VTEX1}}. This implies
%
\begin{equation}
	\label{b_term_bound_VTEX1}
	\begin{split} \|u_{n+1}(\boldsymbol{y}) - u_{n}(\boldsymbol{y})\|_{V}^{2}
		&\leq \frac{\overline{b}}{2} \left \|{u_{n+1}(\boldsymbol{y})-u_{n}(
			\boldsymbol{y})}\right \|_{\mathcal L(m+1/1)} \left \|{u_{n}(
			\boldsymbol{y})-u_{n-1}(\boldsymbol{y})}\right \|_{\mathcal L(m+1/1)}
		\sum _{j=0}^{m-1} \left \|{u_{n}(\boldsymbol{y})}\right \|_{
			\mathcal L(m+1/1)}^{m-1-j}\left \|{u_{n-1}(\boldsymbol{y})}\right \|_{
			\mathcal L(m+1/1)}^{j}
		\\
		&\leq \frac{\overline{b}}{2} \left \|{u_{n+1}(\boldsymbol{y})-u_{n}(
			\boldsymbol{y})}\right \|_{V} \left \|{u_{n}(\boldsymbol{y})-u_{n-1}(
			\boldsymbol{y})}\right \|_{V} \sum _{j=0}^{m-1} \left \|{u_{n}(
			\boldsymbol{y})}\right \|_{V}^{m-1-j}\left \|{u_{n-1}(\boldsymbol{y})}
		\right \|_{V}^{j},
	\end{split}
	%
\end{equation}
where \reftext{\eqref{Sobolev_ineq_VTEX1}} has been used in the last step. Since
$\left \{u_{n}\right \}\subset {\mathcal B}(0,{\overline{f}})$ for all
$n$, the sum in the right-hand side of \reftext{\eqref{b_term_bound_VTEX1}} is bounded
by $m \overline{f}^{m-1}$. This and \reftext{Assumption~\ref{General_b_Assump_VTEX1}} imply
the contraction property
\begin{align*}
	\left \|{u_{n+1}(\boldsymbol{y})-u_{n}(\boldsymbol{y})}\right \|_{V}
	\leq \gamma \left \|{u_{n}(\boldsymbol{y})-u_{n-1}(\boldsymbol{y})}
	\right \|_{V}.
\end{align*}
Since $\gamma <1$, the sequence
$\left \{u_{n}(\boldsymbol{y})\right \}$ converges to a fixed point in
$ {\mathcal B}(0,{\overline{f}})$ by the Banach fixed point theorem.
\end{proof}
Estimate \reftext{\eqref{A1m1}} and \reftext{Lemma~\ref{lem:A1m2}} imply \reftext{\eqref{u-bound}} with
%
\begin{align}
\label{norm_u_bound_VTEX1}
\overline{u}:=
\begin{cases}
	\frac{\overline{f}}{1-\gamma} &\text{ if } m=1,
	\\
	\overline{f} &\text{ if } m\geq 2.
\end{cases} 
\end{align}
According to the \reftext{Assumption~\ref{General_b_Assump_VTEX1}}, the magnitude of
$b$ will decrease as $m$ and $\overline{f}$ grow. As an alternative, we
consider \reftext{Assumption~\ref{Possitive_b_Assump_VTEX1}}, which helps to relax this,
when $T_{\boldsymbol{y}}$ is nonnegative.

\begin{assumption}
\label{Possitive_b_Assump_VTEX1}
The function $b(\boldsymbol{y})$ is non-negative for almost
$(\boldsymbol{x},\boldsymbol{y})\in D\times U$ and $m$ is an odd positive
integer such that $(d,m)\in \mathcal M$.
\end{assumption}
In case \reftext{Assumption~\ref{Possitive_b_Assump_VTEX1}} is satisfied, we choose the
operator $\mathcal S_{\boldsymbol{y}}: V\to V^{*}$ such that
$\left \langle{\mathcal S_{\boldsymbol{y}}(u)},{v} \right \rangle =A_{
\boldsymbol{y}}(u,v)+T_{\boldsymbol{y}}(u,v)$. As an immediate consequence
of \reftext{\eqref{A_bound_VTEX1}} and \reftext{\eqref{T_bound_VTEX1}}, the operator
$\mathcal S_{\boldsymbol{y}}$ is continuous and bounded. Moreover,
$\mathcal S_{\boldsymbol{y}}$ is a strictly monotone operator, since
$b(\boldsymbol{y})\geq 0$ and $(\cdot )^{m}$ is a
monotonically increasing function when $m$ is an odd integer.
Indeed, for every $w,v\in V$ such that $w\neq v$, we have
\begin{align*}
\left \langle{\mathcal S_{\boldsymbol{y}}(w)-\mathcal S_{
		\boldsymbol{y}}(v)},{w-v} \right \rangle = C_{m}^{2}\int _{D} a(
\boldsymbol{y}) \left |\nabla w-\nabla v\right |^{2}+\int _{D} b(
\boldsymbol{y}) (w^{m}-v^{m})(w-v) \geq \left \|{w-v}\right \|_{V}^{2}
>0.
\end{align*}
Substitute $v= 0$ in the above inequality and notice that
$\mathcal S_{\boldsymbol{y}}(0)=0$ to arrive at
%
\begin{align}
\label{coercivity_A_VTEX1}
\left \langle{\mathcal S_{\boldsymbol{y}}(w)},{w} \right \rangle = A_{
	\boldsymbol{y}}(w,w)+T_{\boldsymbol{y}}(w,w) \geq \left \|{w}\right
\|_{V}^{2} \qquad \forall w \in V,
\end{align}
and thus $\mathcal S_{\boldsymbol{y}}$ is coercive. By the Minty-Browder
Theorem, the operator $\mathcal S_{\boldsymbol{y}}$ is bijective, and hence
the problem \reftext{\eqref{variational_form_VTEX1}} has a uniquely determined solution in
$V$. In this case \reftext{\eqref{coercivity_A_VTEX1}} gives
\begin{equation*}
\|u\|_{V}^{2} \leq \left \langle{\mathcal S_{\boldsymbol{y}}(u)},{u}
\right \rangle = \left \langle{f},{u} \right \rangle \leq \|f\|_{V^{*}}
\|u\|_{V}
\end{equation*}
and hence, by \reftext{\eqref{ab-bounds}}, we may choose
$\overline{u}= \overline{f} $ for the upper bound \reftext{\eqref{u-bound}}.

For each $\boldsymbol{y}\in U$, we denote by
$\widetilde A_{\boldsymbol{y}}(u, w,v)$ the linearization of \reftext{\eqref{semi_prob_VTEX1}} mapping 
$ V \times V\times V \rightarrow \mathbb R$ as
%
\begin{align}
\label{Atilde-def}
\widetilde A_{\boldsymbol{y}}(u,w,v)= C_{m}^{2} \int _{D} a(
\boldsymbol{y}) \nabla w\cdot \nabla v + m\int _{D} b(\boldsymbol{y})
\, u^{m-1}\, w \, v.
\end{align}
The following Lemma shows the coercivity of
$\widetilde A_{\boldsymbol{y}}$, which is required for the regularity proof
in Section~\ref{sec:_main_result_VTEX1}.
%
\begin{lemma}%
\label{lem:coercive-type}
Let $a,b$ and $f$ satisfy \reftext{Assumption~\ref{General_b_Assump_VTEX1}} or \reftext{Assumption~\ref{Possitive_b_Assump_VTEX1}}. The operator
$\widetilde A_{\boldsymbol{y}}$ is uniformly coercive in
$\boldsymbol{y}$, i.e.
%
\begin{align}
	\label{coercive_lower_bound_VTEX1}
	\widetilde A_{\boldsymbol{y}}(u, v,v )\geq C_{A} \left \|{v}\right \|_{V}^{2},
	\quad \forall v\in V \text{ and } \forall u \in {\mathcal B}(0,
	\overline{f}),
\end{align}
where $C_{A}:= 1$ if \reftext{Assumption~\ref{Possitive_b_Assump_VTEX1}} holds and
$C_{A}:= 1-\gamma $ if \reftext{Assumption~\ref{General_b_Assump_VTEX1}} holds.
\end{lemma}
\begin{proof}
\reftext{Assumption~\ref{Possitive_b_Assump_VTEX1}} sets $b$ non-negative and $m$ an odd
integer. This implies that $b(\boldsymbol{y}) u^{m-1}$ is nonnegative and
\begin{align*}%
	\widetilde A_{\boldsymbol{y}}(u, v,v )\geq C_{m}^{2}\left \|{v}
	\right \|_{H^{1}_{0}(D)}^{2} + m\int _{D} b(\boldsymbol{y})\, u^{m-1}
	\, v^{2} \geq \left \|{v}\right \|_{V}^{2}.
\end{align*}
This shows that $C_{A} = 1$ in this case. If, instead, \reftext{Assumption~\ref{General_b_Assump_VTEX1}} is valid, analogous considerations imply
\begin{align*}
	\widetilde A_{\boldsymbol{y}}(u, v,v )\geq \left (1- m\,
	\frac{\overline{b\,}}{2}\left \|{u}\right \|_{V}^{m-1}\right )\left
	\|{v}\right \|_{V}^{2}.
\end{align*}
For $m=1$, we have $\overline{b}/2=\gamma $, and therefore
$C_{A}=1-\gamma $. In case $m\geq 2$, notice that
$u \in {\mathcal B}(0,\overline{f})$, we obtain
\begin{align*}
	\widetilde A_{\boldsymbol{y}}(u, v,v ) \geq \left (1-m\left (
	\frac{\gamma }{m \overline{f}^{m-1}}\right ) \overline{f}^{m-1}
	\right )\left \|{v}\right \|_{V}^{2} \geq \left (1-\gamma \right )
	\left \|{v}\right \|_{V}^{2}.
\end{align*}
It also shows that $C_{A}=1-\gamma $ and finishes the proof.
\end{proof}

\section{Parametric regularity}
\label{sec:_main_result_VTEX1}

\subsection{Parametric regularity in $H^{1}(D)$}
\label{sec4.1}

The following theorem is the first main regularity result of this paper.
%
\begin{theorem}%
\label{gevrey_regularity_for_semilinear_VTEX1}
Let the coefficients $a,b$ and the right-hand side $f$ of \reftext{\eqref{variational_form_VTEX1}} satisfy \reftext{Assumption~\ref{Assumption}} for some
$\delta \geq 1$ and suppose moreover that either \reftext{Assumption~\ref{General_b_Assump_VTEX1}} or \reftext{Assumption~\ref{Possitive_b_Assump_VTEX1}} hold. Then
the solution $u$ of \reftext{\eqref{variational_form_VTEX1}} is of class Gevrey-$
\delta $. More precisely, the following estimates are valid for all
$\boldsymbol{\nu }\in
\mathcal F\setminus \left \{\boldsymbol{0}\right
\}$ and $\boldsymbol{y}\in U$
%
\begin{align}
	\label{u_V_norm_ref_VTEX1}
	\left \|{\partial ^{\boldsymbol{\nu }} u( \boldsymbol{y})}\right \|_{V}
	\leq
	\frac{C_{u} \rho ^{\left |\boldsymbol{\nu }\right |-1} \left [\tfrac{1}{2} \right ]_{\left |\boldsymbol{\nu }\right |} }{\boldsymbol{R}^{{\boldsymbol{\nu }}}}
	(\left |\boldsymbol{\nu }\right |!)^{\delta -1}
\end{align}
and
%
\begin{align}
	\label{u_H1_norm_ref_VTEX1}
	\left \|{\partial ^{\boldsymbol{\nu }} u( \boldsymbol{y})}\right \|_{H^{1}_{0}(D)}
	\leq
	\frac{C_{u}\rho ^{\left |\boldsymbol{\nu }\right |-1}(\left |\boldsymbol{\nu }\right |!)^{\delta }}{C_{m}\boldsymbol{R}^{\boldsymbol{\nu }}}.
\end{align}
The constants in the above bounds are explicitly determined as
%
\begin{equation}
	\label{Cu-rho-def}
	C_{u}:=
	C_{A}^{-1}\overline{u}\left (\overline{a}+
	\overline{b}\,\overline{u}^{m-1}+1\right ) \qquad \text{and} \qquad
	\rho :=
	\max \left \{2,
	\frac{2\overline{a}+\overline{b}\left (m\overline{u}^{m-1}+ (3C_{u})^{m-1}(m+1)\right )}{C_{A}}+1
	\right \}.
\end{equation}
\end{theorem}
To prove the above Theorem, we require auxiliary upper bounds for the derivatives
of the solution from \reftext{Lemma~\ref{lem:u-bnd}} and \reftext{Lemma~\ref{upper_bound_for_um_derivative_VTEX1}} below.
%
\begin{lemma}%
\label{lem:u-bnd}
For sufficiently regular solutions of \reftext{\eqref{semi_prob_VTEX1}} there holds
%
\begin{equation}
	\label{u_V_norm_upper_VTEX1}
	\begin{split} C_{A}\left \|{\partial ^{\boldsymbol{\nu }+
				\boldsymbol{e}}u}\right \|_{V} &\leq
		\left \|{\partial ^{\boldsymbol{\nu }+
				\boldsymbol{e}}a}\right \|_{L^{\infty }(D)} \left \|{u}\right \|_{V} +
		\sum _{
			\boldsymbol{0}< \boldsymbol{\eta }\leq
			\boldsymbol{\nu }} \binom{\boldsymbol{\nu }}{\boldsymbol{\eta }}
		\left \|{\partial ^{\boldsymbol{\nu }+\boldsymbol{e}-
				\boldsymbol{\eta }}a}\right \|_{L^{\infty }(D)} \left \|{\partial ^{
				\boldsymbol{\eta }}u}\right \|_{V}
		\\
		&\quad +
		\left \|{\partial ^{\boldsymbol{\nu }+
				\boldsymbol{e}}b}\right \|_{L^{\infty }(D)} \left \|{u^{m}}\right \|_{
			\mathcal L(m+1/m)}+ \sum _{
			\boldsymbol{0}< \boldsymbol{\eta }\leq
			\boldsymbol{\nu }} \binom{\boldsymbol{\nu }}{\boldsymbol{\eta }}
		\left \|{\partial ^{\boldsymbol{\nu }+\boldsymbol{e}-
				\boldsymbol{\eta }}b}\right \|_{L^{\infty }(D)} \left \|{\partial ^{
				\boldsymbol{\eta }}(u^{m})}\right \|_{\mathcal L(m+1/m)}
		\\
		&\quad + \left \|{\partial ^{\boldsymbol{\nu }+\boldsymbol{e}}f}
		\right \|_{V^{*}} + \sum _{\boldsymbol{0}< \boldsymbol{\eta }\leq
			\boldsymbol{\nu }} \binom{\boldsymbol{\nu }}{\boldsymbol{\eta }}
		\left \|{\partial ^{\boldsymbol{\eta }}a}\right \|_{L^{
				\infty }(D)} \left \|{\partial ^{\boldsymbol{\nu }+\boldsymbol{e}-
				\boldsymbol{\eta }}u}\right \|_{V}
		\\
		&\quad
		+ m\sum _{\boldsymbol{0}< \boldsymbol{\eta }\leq
			\boldsymbol{\nu }} \binom{\boldsymbol{\nu }}{\boldsymbol{\eta }}
		\left \|{\partial ^{\boldsymbol{\eta }}b}\right \|_{L^{
				\infty }(D)} \left \|{u^{m-1}}\right \|_{\mathcal L(m+1/m-1)} \|
		\partial ^{\boldsymbol{\nu }+\boldsymbol{e}-\boldsymbol{\eta }} u\|_{
			\mathcal L(m+1/1)}
		\\
		&\quad + m\sum _{\boldsymbol{0}< \boldsymbol{\eta }\leq
			\boldsymbol{\nu }} \sum _{
			\boldsymbol{0}< \boldsymbol{\ell }\leq
			\boldsymbol{\eta }} \binom{\boldsymbol{\nu }}{\boldsymbol{\eta }}
		\binom{\boldsymbol{\eta }}{\boldsymbol{\ell }}
		\left \|{\partial ^{\boldsymbol{\eta }-\boldsymbol{\ell }}b}\right \|_{L^{
				\infty }(D)} \left \|{\partial ^{\boldsymbol{\ell }}(u^{m-1})}\right
		\|_{\mathcal L(m+1/m-1)} \| \partial ^{\boldsymbol{\nu }+
			\boldsymbol{e}-\boldsymbol{\eta }} u\|_{\mathcal L(m+1/1)},
	\end{split}
	%
\end{equation}
where $\boldsymbol{e}$ is a unit multi-index in
$\mathcal F$, i.e. $\left |\boldsymbol{e}\right |=1$.
\end{lemma}
\begin{proof}
We recall the variational formulation \reftext{\eqref{variational_form_VTEX1}} and take
the $\boldsymbol{e}$-th derivative of both sides with respect
to $ \boldsymbol{y}$. Collecting the terms with
$\partial ^{\boldsymbol{e}} u$ on the left-hand side we obtain
%
\begin{equation}
	\label{ve_derivative_VTEX1}
	\begin{split} C_{m}^{2}\int _{D} a\, \partial ^{\boldsymbol{e}}\nabla u
		\cdot \nabla v + m\int _{D} b\, u^{m-1} \partial ^{\boldsymbol{e}}u\, v
		= - C_{m}^{2}\int _{D} \partial ^{\boldsymbol{e}} a \nabla u\cdot
		\nabla v - \int _{D} \partial ^{\boldsymbol{e}} b\, u^{m} v + C_{m}
		\int _{D} \partial ^{\boldsymbol{e}} f v .
	\end{split}
	%
\end{equation}
Observe that the left-hand side can be expressed as
$\widetilde A_{\boldsymbol{y}}(u,\partial ^{\boldsymbol{e}} u,v)$, where
the linearized form $\widetilde A_{\boldsymbol{y}}$ has been introduced
in \reftext{\eqref{Atilde-def}}. Notice that the first and the second argument of
this expression depend on $\boldsymbol{y}$. Therefore, if we take higher
$\boldsymbol{\nu }$-th order derivatives of \reftext{\eqref{ve_derivative_VTEX1}}, both
the first and the second argument of $\widetilde A_{\boldsymbol{y}}$ will
generate further terms by the Leibniz product rule. But the highest order
derivative $\partial ^{\boldsymbol{\nu }+ \boldsymbol{e}} u$ will only
appear in the term
$\widetilde A_{\boldsymbol{y}}(u,\partial ^{\boldsymbol{\nu }+
	\boldsymbol{e}} u,v)$. Isolating this term on the left-hand side, we obtain
%
\begin{equation}
	\label{vnu+ve_derivative_VTEX1}
	\begin{split} \widetilde A_{\boldsymbol{y}}(u,\partial ^{
			\boldsymbol{\nu }+\boldsymbol{e}} u,v) =& - C_{m}^{2} \sum _{
			\boldsymbol{0}\leq \boldsymbol{\eta }\leq \boldsymbol{\nu }} \binom{\boldsymbol{\nu }}{\boldsymbol{\eta }}  \int _{D} \partial ^{
			\boldsymbol{\nu }+\boldsymbol{e}-\boldsymbol{\eta }} a\, \partial ^{
			\boldsymbol{\eta }}\nabla u\cdot \nabla v - \sum _{\boldsymbol{0}
			\leq \boldsymbol{\eta }\leq \boldsymbol{\nu }} \binom{\boldsymbol{\nu }}{\boldsymbol{\eta }} \int _{D} \partial ^{
			\boldsymbol{\nu }+\boldsymbol{e}-\boldsymbol{\eta }} b\, \partial ^{
			\boldsymbol{\eta }} (u^{m}) v+ C_{m}\int _{D} \partial ^{
			\boldsymbol{\nu }+\boldsymbol{e}} f v
		\\
		& - C_{m}^{2} \sum _{\boldsymbol{0}< \boldsymbol{\eta }\leq
			\boldsymbol{\nu }} \binom{\boldsymbol{\nu }}{\boldsymbol{\eta }}
		\int _{D} \partial ^{\boldsymbol{\eta }} a\, \partial ^{
			\boldsymbol{\nu }+\boldsymbol{e}-\boldsymbol{\eta }}\nabla u\cdot
		\nabla v - m\sum _{\boldsymbol{0}< \boldsymbol{\eta }\leq
			\boldsymbol{\nu }} \binom{\boldsymbol{\nu }}{\boldsymbol{\eta }}
		\sum _{\boldsymbol{0}\leq \boldsymbol{\ell }\leq
			\boldsymbol{\eta }} \binom{\boldsymbol{\eta }}{\boldsymbol{\ell }}
		\int _{D} \partial ^{\boldsymbol{\eta }-\boldsymbol{\ell }} b
		\, \partial ^{\boldsymbol{\ell }} (u^{m-1}) \partial ^{
			\boldsymbol{\nu }+\boldsymbol{e}-\boldsymbol{\eta }} u\, v.
	\end{split}
	%
\end{equation}
Since \reftext{\eqref{vnu+ve_derivative_VTEX1}} is valid for all $v \in V$ we may select
specifically $v=\partial ^{\boldsymbol{\nu }+\boldsymbol{e}} u$. According
to \reftext{Lemma~\ref{lem:coercive-type}}, the left-hand side admits the bound
$\widetilde A_{\boldsymbol{y}}(u,\partial ^{\boldsymbol{\nu }+
	\boldsymbol{e}} u,\partial ^{\boldsymbol{\nu }+\boldsymbol{e}} u)
\geq C_{A} \left \|{ \partial ^{\boldsymbol{\nu }+\boldsymbol{e}} u}
\right \|_{ V}^{2}$. Applying the triangle and the Cauchy-Schwarz inequality,
we get the estimate
\begin{equation*}
	\label{vnu+ve_derivative2_VTEX1}
	\begin{split} C_{A} &\left \|{ \partial ^{\boldsymbol{\nu }+
				\boldsymbol{e}} u}\right \|_{V} \leq \sum _{\boldsymbol{0}\leq
			\boldsymbol{\eta }\leq \boldsymbol{\nu }} \binom{\boldsymbol{\nu }}{\boldsymbol{\eta }}  \|\partial ^{\boldsymbol{\nu }+
			\boldsymbol{e}-\boldsymbol{\eta }} a\|_{L^{\infty}(D)}\, \|\partial ^{
			\boldsymbol{\eta }} u\|_{V} + \sum _{\boldsymbol{0}\leq
			\boldsymbol{\eta }\leq \boldsymbol{\nu }} \binom{\boldsymbol{\nu }}{\boldsymbol{\eta }} \| \partial ^{\boldsymbol{\nu }+
			\boldsymbol{e}-\boldsymbol{\eta }} b\|_{L^{\infty}(D)}\, \|\partial ^{
			\boldsymbol{\eta }} (u^{m})\|_{ \mathcal L(m+1/m) }+ \|\partial ^{
			\boldsymbol{\nu }+\boldsymbol{e}} f \|_{V^{*}}
		\\
		+& \sum _{\boldsymbol{0}< \boldsymbol{\eta }\leq \boldsymbol{\nu }}
		\binom{\boldsymbol{\nu }}{\boldsymbol{\eta }} \| \partial ^{
			\boldsymbol{\eta }} a\|_{L^{\infty}(D)}\, \|\partial ^{
			\boldsymbol{\nu }+\boldsymbol{e}-\boldsymbol{\eta }} u\|_{V} + m\sum _{
			\boldsymbol{0}< \boldsymbol{\eta }\leq \boldsymbol{\nu }} \sum _{
			\boldsymbol{0}\leq \boldsymbol{\ell }\leq \boldsymbol{\eta }} \binom{\boldsymbol{\nu }}{\boldsymbol{\eta }}  \binom{\boldsymbol{\eta }}{\boldsymbol{\ell }} \|\partial ^{
			\boldsymbol{\eta }-\boldsymbol{\ell }} b\|_{L^{\infty}(D)}\, \|
		\partial ^{\boldsymbol{\ell }} (u^{m-1})\|_{ \mathcal L(m+1/m-1) } \,
		\| \partial ^{\boldsymbol{\nu }+\boldsymbol{e}-\boldsymbol{\eta }} u
		\|_{ \mathcal L(m+1/1)} ,
	\end{split}
\end{equation*}
where we have applied the H\"older inequality \reftext{\eqref{adjusted_Hoelder_VTEX1}} and the Sobolev embedding estimate \reftext{\eqref{Sobolev_ineq_VTEX1}} for the terms on the right-hand side. Notice that
$\|\partial ^{\boldsymbol{\nu }+\boldsymbol{e}} u\|_{V}$ is cancelled on
the both sides.
We conclude the proof by separating the terms with
$\boldsymbol{\eta }=\boldsymbol{0}$ in the first two sums and the term
with $\boldsymbol{\ell }=\boldsymbol{\eta }$ in the last sum.
\end{proof}
\begin{remark}
\label{rem4.3}
The existence and uniqueness of the partial derivatives
$\partial ^{\boldsymbol{\nu }+\boldsymbol{e}}u$ at some
$\boldsymbol{y}\in U$ for all $\boldsymbol{\nu }\in \mathcal F$ follows
by induction. Indeed, for the inductive step, the equation \reftext{\eqref{vnu+ve_derivative_VTEX1}} illustrates that
$\partial ^{\boldsymbol{\nu }+\boldsymbol{e}}u$ is determined by its lower
order derivatives. Thus, when combined with \reftext{Lemma~\ref{lem:coercive-type}}, it confirms the well-posedness of
$\partial ^{\boldsymbol{\nu }+\boldsymbol{e}}u$.
\end{remark}%
The right-hand side of \reftext{\eqref{u_V_norm_upper_VTEX1}} contains the terms of the
type
$\left \|{\partial ^{\boldsymbol{\mu }}(u^{k})}\right \|_{\mathcal L(m+1/k)}$,
where $k = m-1$ or $m$ and
$\boldsymbol{\mu }\in
\mathcal F\setminus \left \{\boldsymbol{0}\right
\}$. The following result is a corollary of the general \reftext{Lemma~\ref{upper_bound_for_f_derivative_VTEX1}} and determines explicit upper bounds
for these powers of $u$, if corresponding bounds for $u$ are available.
This result together with \reftext{Lemma~\ref{lem:u-bnd}} is the key ingredient in
the inductive proof of \reftext{Theorem~\ref{gevrey_regularity_for_semilinear_VTEX1}}.
%
\begin{corollary}%
\label{upper_bound_for_um_derivative_VTEX1}
Let
$u(\boldsymbol{y})\in {\mathcal B}(0,\overline{u})$
and
$\boldsymbol{\mu }\in
\mathcal F\setminus \left \{{\boldsymbol{0}}
\right \}$, suppose that \reftext{\eqref{u_V_norm_ref_VTEX1}} holds for all multi-index
$\boldsymbol{\ell }\in \mathcal F\setminus \left
\{{\boldsymbol{0}}\right \}$ and
$\boldsymbol{\ell }\leq \boldsymbol{\mu }$, i.e.
%
\begin{align}
	\label{uV_deri_VTEX1}
	\left \|{\partial ^{\boldsymbol{\ell }} u (\boldsymbol{y})}\right \|_{V}
	\leq
	\frac{C_{u} \rho ^{\left |\boldsymbol{\ell }\right |-1} \left [\tfrac{1}{2} \right ]_{\left |\boldsymbol{\ell }\right |} }{\boldsymbol{R}^{{\boldsymbol{\ell }}}}
	(\left |\boldsymbol{\ell }\right |!)^{\delta -1}
	\quad \forall \boldsymbol{\ell }\leq
	\boldsymbol{\mu }\text{ and } \left |\boldsymbol{\ell }\right |\geq 1.
\end{align}
Then, for all positive integers $q\leq m+1$ the following estimates are
valid
%
\begin{equation}
	\label{um_derivative_VTEX1}
	\left \|{\partial ^{\boldsymbol{\mu }}(u(\boldsymbol{y})^{q})}\right
	\|_{\mathcal L(m+1/q)} \leq
	\frac{(3^{q-1}C_{u}^{q})\, \rho ^{\left |\boldsymbol{\mu }\right |-1} \left [\tfrac{1}{2} \right ]_{\left |\boldsymbol{\mu }\right |} }{\boldsymbol{R}^{{\boldsymbol{\mu }}}}
	(\left |\boldsymbol{\mu }\right |!)^{\delta -1}
\end{equation}
as long as $(d,m) \in \mathcal M$, i.e. $H^{1}(D)$ is continuously embedded
into $L^{m+1}(D)$.
\end{corollary}
\begin{proof}
Notice that $\overline{u}\leq C_{u}$, implying 
$u(\boldsymbol{y})\in {\mathcal B}(0,\overline{u})\subset {\mathcal B}(0,C_{u})$.
We also recall the definition of $\rho $ in \reftext{\eqref{Cu-rho-def}} implying
$\rho \geq 2$. If $(d,m) \in \mathcal M$, we have
$\|u\|_{L^{m+1}(D)} \leq \|u\|_{V}$, therefore the statement of the
lemma follows from \reftext{Lemma~\ref{upper_bound_for_f_derivative_VTEX1}} with
$C_{u}$, $\rho $ and $\boldsymbol{R}/\rho $ in place
of $C_{f}$, $\rho $ and $\boldsymbol{R}$ respectively.
\end{proof}
\begin{proof}[Proof of \reftext{Theorem~\ref{gevrey_regularity_for_semilinear_VTEX1}}]
Observe that \reftext{\eqref{u_H1_norm_ref_VTEX1}} is a simple corollary from \reftext{\eqref{u_V_norm_ref_VTEX1}} by changing from the $V$-norm to the
$H^{1}_{0}(D)$-norm and the trivial bound
$\left [\tfrac{1}{2} \right ]_{n} \leq n!$. Therefore it remains to prove \reftext{\eqref{u_V_norm_ref_VTEX1}}. Here we argue by induction with respect to the order
of the derivative $\boldsymbol{\nu }$. For the first-order derivatives
we use \reftext{\eqref{u_V_norm_upper_VTEX1}} with
$\boldsymbol{\nu }= \boldsymbol{0}$ and get
\begin{align*}
	C_{A}\left \|{\partial ^{\boldsymbol{e}}u}\right \|_{V} \leq \left \|{
		\partial ^{\boldsymbol{e}} a}\right \|_{L^{\infty }(D)}\left \|{u}
	\right \|_{V} + \left \|{\partial ^{\boldsymbol{e}} b}\right \|_{L^{
			\infty }(D)}\left \|{u^{m}}\right \|_{\mathcal L(m+1/m)} + \left \|{
		\partial ^{\boldsymbol{e}} f}\right \|_{V^{*}}.
\end{align*}
For the term with $u^{m}$ we recall the H\"older estimate \reftext{\eqref{multiple_Hoelder_VTEX1}}, the Sobolev embedding \reftext{\eqref{Sobolev_ineq_VTEX1}}, and \reftext{\eqref{u-bound}} to obtain the upper bound
$\left \|{u^{m}}\right \|_{\mathcal L(m+1/m)} \leq \left \|{u}\right
\|_{\mathcal L(m+1/1)}^{m} \leq \|u\|_{V}^{m}\leq \overline{u}^{m}$. Using
this and the regularity assumption \reftext{\eqref{abf_assumption_VTEX1}}, we derive
\begin{align*}
	\left \|{\partial ^{\boldsymbol{e}}u}\right \|_{V} \leq
	\frac{\overline{a} \left [\tfrac{1}{2} \right ]_{1} }{C_{A}\boldsymbol{R}^{\boldsymbol{e}}}
	\, \overline{u} +
	\frac{\overline{b} \left [\tfrac{1}{2} \right ]_{1} }{C_{A}\boldsymbol{R}^{\boldsymbol{e}}}
	\, \overline{u}^{m} +
	\frac{\overline{f} \left [\tfrac{1}{2} \right ]_{1} }{C_{A}\boldsymbol{R}^{\boldsymbol{e}}}
	= \overline{u}\left (\overline{a}+\overline{b}\,\overline{u}^{m-1}+1
	\right )
	\frac{\left [\tfrac{1}{2} \right ]_{1} }{C_{A}\boldsymbol{R}^{\boldsymbol{e}}}
	\leq C_{u}
	\frac{\left [\tfrac{1}{2} \right ]_{1} }{\boldsymbol{R}^{\boldsymbol{e}}}.
\end{align*}
Thus, the base of induction is satisfied for the constant $C_{u}$ defined
in \reftext{\eqref{Cu-rho-def}}. Suppose now that \reftext{\eqref{u_V_norm_ref_VTEX1}} is valid for
the $\boldsymbol{\nu }$-th derivative. Our aim is to show that the same
bound holds for the $(\boldsymbol{\nu }+ \boldsymbol{e})$-th order derivative,
where $\left |\boldsymbol{\nu }\right |\geq 1$ and $\boldsymbol{e}$ is
a unit multi-index. For this we combine \reftext{\eqref{u_V_norm_upper_VTEX1}} with regularity assumptions \reftext{\eqref{abf_assumption_VTEX1}}, the inductive assumption, and \reftext{\eqref{um_derivative_VTEX1}} for $\boldsymbol{\mu }\leq \boldsymbol{\nu }$ (this
is valid by \reftext{Corollary~\ref{upper_bound_for_um_derivative_VTEX1}} and the inductive
assumption) to arrive at
\begin{align*}
	C_{A}\left \|{\partial ^{\boldsymbol{\nu }+\boldsymbol{e}}u}\right \|_{V}
	&\leq
	\frac{\overline{u}\,\overline{a} \left [\tfrac{1}{2} \right ]_{\left |\boldsymbol{\nu }+\boldsymbol{e}\right |} }{\boldsymbol{R}^{\boldsymbol{\nu }+\boldsymbol{e}}}(
	\left |\boldsymbol{\nu }+\boldsymbol{e}\right |!)^{\delta -1} +
	\sum _{
		\boldsymbol{0}< \boldsymbol{\eta }\leq
		\boldsymbol{\nu }} \binom{\boldsymbol{\nu }}{\boldsymbol{\eta }}
	\frac{\overline{a} \left [\tfrac{1}{2} \right ]_{\left |\boldsymbol{\nu }+\boldsymbol{e}-\boldsymbol{\eta }\right |} }{\boldsymbol{R}^{\boldsymbol{\nu }+\boldsymbol{e}-\boldsymbol{\eta }}}(
	\left |\boldsymbol{\nu }+\boldsymbol{e}-\boldsymbol{\eta }\right |!)^{
		\delta -1}\,
	\frac{C_{u} \rho ^{\left |\boldsymbol{\eta }\right |-1} \left [\tfrac{1}{2} \right ]_{\left |\boldsymbol{\eta }\right |} }{\boldsymbol{R}^{{\boldsymbol{\eta }}}}
	(\left |\boldsymbol{\eta }\right |!)^{\delta -1}
	\\
	&\quad
	+
	\frac{\overline{u}^{m}\overline{b} \left [\tfrac{1}{2} \right ]_{\left |\boldsymbol{\nu }+\boldsymbol{e}\right |} }{\boldsymbol{R}^{\boldsymbol{\nu }+\boldsymbol{e}}}(
	\left |\boldsymbol{\nu }+\boldsymbol{e}\right |!)^{\delta -1} +
	\sum _{
		\boldsymbol{0}<\boldsymbol{\eta }\leq
		\boldsymbol{\nu }} \binom{\boldsymbol{\nu }}{\boldsymbol{\eta }}
	\frac{\overline{b} \left [\tfrac{1}{2} \right ]_{\left |\boldsymbol{\nu }+\boldsymbol{e}-\boldsymbol{\eta }\right |} }{\boldsymbol{R}^{\boldsymbol{\nu }+\boldsymbol{e}-\boldsymbol{\eta }}}(
	\left |\boldsymbol{\nu }+\boldsymbol{e}-\boldsymbol{\eta }\right |!)^{
		\delta -1}\,
	\frac{3^{m-1}C_{u}^{m} \,\rho ^{\left |\boldsymbol{\eta }\right |-1} \left [\tfrac{1}{2} \right ]_{\left |\boldsymbol{\eta }\right |} }{\boldsymbol{R}^{{\boldsymbol{\eta }}}}(
	\left |\boldsymbol{\eta }\right |!)^{\delta -1}
	\\
	&\quad +
	\frac{\overline{f} \left [\tfrac{1}{2} \right ]_{\left |\boldsymbol{\nu }+\boldsymbol{e}\right |} }{\boldsymbol{R}^{\boldsymbol{\nu }+\boldsymbol{e}}}
	(\left |\boldsymbol{\nu }+\boldsymbol{e}\right |!)^{\delta -1} +
	\sum _{\boldsymbol{0}< \boldsymbol{\eta }\leq \boldsymbol{\nu }}
	\binom{\boldsymbol{\nu }}{\boldsymbol{\eta }}
	\frac{\overline{a} \left [\tfrac{1}{2} \right ]_{\left |\boldsymbol{\eta }\right |} }{\boldsymbol{R}^{\boldsymbol{\eta }}}(
	\left |\boldsymbol{\eta }\right |!)^{\delta -1}\,
	\frac{C_{u} \rho ^{\left |\boldsymbol{\nu }+\boldsymbol{e}-\boldsymbol{\eta }\right |-1} \left [\tfrac{1}{2} \right ]_{\left |\boldsymbol{\nu }+\boldsymbol{e}-\boldsymbol{\eta }\right |} }{\boldsymbol{R}^{{\boldsymbol{\nu }+\boldsymbol{e}-\boldsymbol{\eta }}}}(
	\left |\boldsymbol{\nu }+\boldsymbol{e}-\boldsymbol{\eta }\right |!)^{
		\delta -1}
	\\
	&\quad
	+ m\sum _{\boldsymbol{0}< \boldsymbol{\eta }\leq
		\boldsymbol{\nu }} \binom{\boldsymbol{\nu }}{\boldsymbol{\eta }}
	\frac{\overline{u}^{m-1}\overline{b} \left [\tfrac{1}{2} \right ]_{\left |\boldsymbol{\eta }\right |} }{\boldsymbol{R}^{\boldsymbol{\eta }}}(
	\left |\boldsymbol{\eta }\right |!)^{\delta -1}
	\frac{C_{u}\, \rho ^{\left |\boldsymbol{\nu }+\boldsymbol{e}-\boldsymbol{\eta }\right |-1} \left [\tfrac{1}{2} \right ]_{\left |\boldsymbol{\nu }+\boldsymbol{e}-\boldsymbol{\eta }\right |} }{\boldsymbol{R}^{{\boldsymbol{\nu }+\boldsymbol{e}-\boldsymbol{\eta }}}}(
	\left |\boldsymbol{\nu }+\boldsymbol{e}-\boldsymbol{\eta }\right |!)^{
		\delta -1}
	\\
	&\quad + m\sum _{\boldsymbol{0}< \boldsymbol{\eta }\leq
		\boldsymbol{\nu }} \sum _{
		\boldsymbol{0}< \boldsymbol{\ell }\leq
		\boldsymbol{\eta }} \binom{\boldsymbol{\nu }}{\boldsymbol{\eta }}
	\binom{\boldsymbol{\eta }}{\boldsymbol{\ell }} 
	\frac{\overline{b} \left [\tfrac{1}{2} \right ]_{\left |\boldsymbol{\eta }-\boldsymbol{\ell }\right |} }{\boldsymbol{R}^{\boldsymbol{\eta }-\boldsymbol{\ell }}\, (\left |\boldsymbol{\eta }-\boldsymbol{\ell }\right |!)^{1-\delta}}
	\,
	\frac{3^{m-2} {C_{u}}^{m-1}\,\rho ^{\left |\boldsymbol{\ell }\right |-1} \left [\tfrac{1}{2} \right ]_{\left |\boldsymbol{\ell }\right |} }{\boldsymbol{R}^{{\boldsymbol{\ell }}}\, (\left |\boldsymbol{\ell }\right |!)^{1-\delta}}
	\frac{C_{u}\, \rho ^{\left |\boldsymbol{\nu }+\boldsymbol{e}-\boldsymbol{\eta }\right |-1}
		\left [\tfrac{1}{2} \right ]_{\left |\boldsymbol{\nu }+\boldsymbol{e}-\boldsymbol{\eta }\right |} }{\boldsymbol{R}^{{\boldsymbol{\nu }+\boldsymbol{e}-\boldsymbol{\eta }}}\,(\left |\boldsymbol{\nu }+\boldsymbol{e}-\boldsymbol{\eta }\right |!)^{1-\delta}}.
\end{align*}
Bound \reftext{\eqref{multiindex-est-1}} yields estimates for products of the factorial
terms. Observe that $0 < C_{A} \leq 1$ and therefore $\rho \geq 2$. This
helps to extract common factors on the right-hand side and obtain
\begin{align*}
	C_{A}\left \|{\partial ^{\boldsymbol{\nu }+\boldsymbol{e}}u}\right \|_{V}
	&\leq
	\frac{\rho ^{\left |\boldsymbol{\nu }\right |-1}(\left |\boldsymbol{\nu }+\boldsymbol{e}\right |!)^{\delta -1}}{ \boldsymbol{R}^{\boldsymbol{\nu }+\boldsymbol{e}}}
	\bigg(
	\overline{u}\, \overline{a}\left [\tfrac{1}{2}
	\right ]_{\left |\boldsymbol{\nu }+\boldsymbol{e}\right |} +C_{u} \,
	\overline{a} \sum _{
		\boldsymbol{0}< \boldsymbol{\eta }\leq
		\boldsymbol{\nu }} \binom{\boldsymbol{\nu }}{\boldsymbol{\eta }}
	\left [\tfrac{1}{2} \right ]_{\left |\boldsymbol{\nu }+
		\boldsymbol{e}-\boldsymbol{\eta }\right |} \left [\tfrac{1}{2}
	\right ]_{\left |\boldsymbol{\eta }\right |}
	\\
	&\quad +
	\overline{u}^{m}\, \overline{b}\left [\tfrac{1}{2}
	\right ]_{\left |\boldsymbol{\nu }+\boldsymbol{e}\right |} +
	3^{m-1} C_{u}^{m} \overline{b}\, \sum _{
		\boldsymbol{0}< \boldsymbol{\eta }\leq
		\boldsymbol{\nu }} \binom{\boldsymbol{\nu }}{\boldsymbol{\eta }}
	\left [\tfrac{1}{2} \right ]_{\left |\boldsymbol{\nu }+
		\boldsymbol{e}-\boldsymbol{\eta }\right |} \left [\tfrac{1}{2}
	\right ]_{\left |\boldsymbol{\eta }\right |} + \overline{f}\left [
	\tfrac{1}{2} \right ]_{\left |\boldsymbol{\nu }+\boldsymbol{e}\right |}
	\\
	&\quad + \overline{a}C_{u} \sum _{\boldsymbol{0}< \boldsymbol{\eta }
		\leq \boldsymbol{\nu }} \binom{\boldsymbol{\nu }}{\boldsymbol{\eta }} \left [\tfrac{1}{2} \right ]_{\left |
		\boldsymbol{\nu }+\boldsymbol{e}-\boldsymbol{\eta }\right |} \left [
	\tfrac{1}{2} \right ]_{\left |\boldsymbol{\eta }\right |}
	+m\,\overline{u}^{m-1}C_{u}\, \overline{b} \sum _{
		\boldsymbol{0}< \boldsymbol{\eta }\leq \boldsymbol{\nu }} \binom{\boldsymbol{\nu }}{\boldsymbol{\eta }} \left [\tfrac{1}{2}
	\right ]_{\left |\boldsymbol{\nu }+\boldsymbol{e}-\boldsymbol{\eta }
		\right |} \left [\tfrac{1}{2} \right ]_{\left |\boldsymbol{\eta }
		\right |}
	\\
	&\quad + m\, 3^{m-2}C_{u}^{m}
	\overline{b}\, \sum _{\boldsymbol{0}< \boldsymbol{\eta }\leq
		\boldsymbol{\nu }} \sum _{
		\boldsymbol{0}< \boldsymbol{\ell }\leq
		\boldsymbol{\eta }} \binom{\boldsymbol{\nu }}{\boldsymbol{\eta }}
	\binom{\boldsymbol{\eta }}{\boldsymbol{\ell }}
	\left [\tfrac{1}{2} \right ]_{\left |\boldsymbol{\eta }-
		\boldsymbol{\ell }\right |} \left [\tfrac{1}{2} \right ]_{\left |
		\boldsymbol{\ell }\right |} \left [\tfrac{1}{2} \right ]_{\left |
		\boldsymbol{\nu }+\boldsymbol{e}-\boldsymbol{\eta }\right |} \bigg).
\end{align*}
According to \reftext{\eqref{multiindex-est-6}} and \reftext{\eqref{multiindex-est-5}}, the
bound for the first four sums is
$[\tfrac{1}{2}]_{|\boldsymbol{\nu }+\boldsymbol{e}|}$, the last sum bounded
by
$3 [\tfrac{1}{2}]_{|\boldsymbol{\nu }+
	\boldsymbol{e}|}$. Recalling that $\overline{f} \leq \overline{u}$ we arrive
at
\begin{align*}
	\left \|{\partial ^{\boldsymbol{\nu }+\boldsymbol{e}}u}\right \|_{V}&
	\leq
	\frac{\overline{u}\left (\overline{a}+\overline{b}\,\overline{u}^{m-1}+1\right )+ 2\overline{a}C_{u} + m\overline{b} \,\overline{u}^{m-1} C_{u} + \overline{b}C_{u}^{m} \left (3^{m-1}+3m\,3^{m-2}\right ) }{C_{u} C_{A}}
	\,
	\frac{	C_{u}\,\rho ^{\left |\boldsymbol{\nu }\right |-1}(\left |\boldsymbol{\nu }+\boldsymbol{e}\right |!)^{\delta -1}}{\boldsymbol{R}^{\boldsymbol{\nu }+\boldsymbol{e}}}
	\left [\tfrac{1}{2} \right ]_{\left |\boldsymbol{\nu }+\boldsymbol{e}
		\right |}
	\\
	& \leq
	\frac{	C_{u}\,\rho ^{\left |\boldsymbol{\nu }\right |}(\left |\boldsymbol{\nu }+\boldsymbol{e}\right |!)^{\delta -1}}{\boldsymbol{R}^{\boldsymbol{\nu }+\boldsymbol{e}}}
	\left [\tfrac{1}{2} \right ]_{\left |\boldsymbol{\nu }+\boldsymbol{e}
		\right |} ,
\end{align*}
where we have used the definition \reftext{\eqref{Cu-rho-def}} of $\rho $ in the
last step. This completes the inductive argument and thereby
the proof of the theorem.
\end{proof}

\subsection{Parametric regularity in $H^{2}(D)$ and pointwise estimates}
\label{sec4.2}

If $D$ is a convex Lipschitz domain and the coefficients $a,b$ and the
right-hand side $f$ are sufficiently regular, it is actually possible to
prove that estimates of the type \reftext{\eqref{u_H1_norm_ref_VTEX1}} remain valid in
$H^{2}(D)$, and therefore in every space containing $H^{2}(D)$ with continuous
embedding. In particular, pointwise estimates follow from
$L^{\infty}(D)$ bounds.

If the data are sufficiently regular (see \reftext{Assumption~\ref{Assumption1}} below for the precise regularity assumptions), formal
integration by parts in the variational formulation \reftext{\eqref{variational_form_VTEX1}} delivers
\begin{equation*}
C_{m}^{2} a \Delta u = b u^{m} - C_{m}^{2}\nabla a
\cdot \nabla u - C_{m} f.
\end{equation*}
Since the right hand {side belongs to} $L^{2}(D)$, it shows
that $\Delta u\in L^{2}(D)$. The Leibniz product rule implies
%
\begin{equation}
\label{Delta_u_variational_form_VTEX1}
C_{m}^{2} a (\Delta \partial ^{\boldsymbol{\nu }} u) =
\partial ^{\boldsymbol{\nu }}(b u^{m} - C_{m}^{2}\nabla a \cdot
\nabla u - C_{m}f) - C_{m}^{2}\sum _{\boldsymbol{0}\leq
	\boldsymbol{\eta }< \boldsymbol{\nu }}
\binom{\boldsymbol{\nu }}{\boldsymbol{\eta }} \partial ^{
	\boldsymbol{\nu }-\boldsymbol{\eta }} a (\Delta \partial ^{
	\boldsymbol{\eta }} u).
\end{equation}
The recurrent equation above implies that we can determine
$\Delta \partial ^{\boldsymbol{\nu }} u$ based on
$\Delta \partial ^{\boldsymbol{\eta }} u$ with
$\eta <\boldsymbol{\nu }$, thereby ensuring that
$\Delta \partial ^{\boldsymbol{\nu }} u \in L^{2}(D)$. Apply the triangle
and the Cauchy-Schwarz inequalities and recall that $a\geq 1$, and hence
%
\begin{equation}
\label{vnu-Delta-u}
C_{m}^{2}\|\Delta \partial ^{\boldsymbol{\nu }} u\|_{L^{2}(D)}
\leq \|C_{m}^{2} a (\Delta \partial ^{
	\boldsymbol{\nu }} u)\|_{L^{2}(D)} \leq \|\partial ^{
	\boldsymbol{\nu }} ( b u^{m} - C_{m}^{2}\nabla a \cdot \nabla u - C_{m}
f)\|_{L^{2}(D)} + C_{m}^{2}\sum _{\boldsymbol{0}\leq
	\boldsymbol{\eta }< \boldsymbol{\nu }}
\binom{\boldsymbol{\nu }}{\boldsymbol{\eta }} \|\partial ^{
	\boldsymbol{\nu }-\boldsymbol{\eta }} a\|_{L^{\infty}(D)} \|\Delta
\partial ^{\boldsymbol{\eta }} u\|_{L^{2}(D)}.
\end{equation}
The second term admits Gevrey-$\delta $ bounds if $L^{\infty}(D)$ and
$V^{*}$-norms for $a$ and $f$ in \reftext{Assumption~\ref{Assumption}} are replaced
by the $W^{1,\infty}(D)$ and $L^{2}(D)$-norms respectively. More precisely,
we replace \reftext{Assumption~\ref{Assumption}} by the following.
%
\begin{assumption}
\label{Assumption1}
For all fixed values $\boldsymbol{y}\in U \in \mathbb R^{s}$ with
$s \in \mathbb N$, the coefficients
$a(\boldsymbol{y}) \in W^{1,\infty}(D)$,
$b(\boldsymbol{y}) \in L^{\infty}(D)$ and $f(y)\in L^{2}(D)$. Moreover,
for all $\boldsymbol{\nu }\in \mathbb N^{s}$ there exist positive constants
$\boldsymbol{R}=(R_{1},R_{2},\dots )$ independent of $s$ such that
\begin{align*}
	\left \|{\partial ^{\boldsymbol{\nu }} a(\boldsymbol{y})}\right \|_{W^{1,
			\infty }(D)} \leq \frac{\overline{a}}{2}
	\frac{(\left |\boldsymbol{\nu }\right |!)^{\delta }}{(2\boldsymbol{R})^{\boldsymbol{\nu }}}
	, \qquad \left \|{\partial ^{\boldsymbol{\nu }} b(\boldsymbol{y})}
	\right \|_{L^{\infty }(D)} \leq \frac{\overline{b}}{2}
	\frac{(\left |\boldsymbol{\nu }\right |!)^{\delta }}{(2\boldsymbol{R})^{\boldsymbol{\nu }}}
	, \qquad \left \|{\partial ^{\boldsymbol{\nu }} f(\boldsymbol{y})}
	\right \|_{L^{2}(D)} \leq \frac{\overline{f}}{2}
	\frac{(\left |\boldsymbol{\nu }\right |!)^{\delta }}{(2\boldsymbol{R})^{\boldsymbol{\nu }}}
	.
\end{align*}
Assume moreover, that $m \in \mathbb N$ is such that $H^{1}(D)$ is continuously
embedded into $L^{2m}(D)$, i.e. $(d,2m-1) \in \mathcal M$, see \reftext{\eqref{cond:_m_d_restriction_VTEX1}}.
\end{assumption}
According to \reftext{\eqref{ff-estimates}} \reftext{Assumption~\ref{Assumption1}} implies
%
\begin{align}
\label{abf1_assumption_VTEX1}
\left \|{\partial ^{\boldsymbol{\nu }} a(\boldsymbol{y})}\right \|_{W^{1,
		\infty }(D)} \leq
\frac{\overline{a} \left [\tfrac{1}{2} \right ]_{\left |\boldsymbol{\nu }\right |} }{\boldsymbol{R}^{\boldsymbol{\nu }}}
(\left |\boldsymbol{\nu }\right |!)^{\delta -1} , \qquad \left \|{
	\partial ^{\boldsymbol{\nu }} b(\boldsymbol{y})}\right \|_{L^{\infty }(D)}
\leq
\frac{\overline{b} \left [\tfrac{1}{2} \right ]_{\left |\boldsymbol{\nu }\right |} }{\boldsymbol{R}^{\boldsymbol{\nu }}}
(\left |\boldsymbol{\nu }\right |!)^{\delta -1} , \qquad \left \|{
	\partial ^{\boldsymbol{\nu }} f(\boldsymbol{y})}\right \|_{L^{2}(D)}
\leq
\frac{\overline{f} \left [\tfrac{1}{2} \right ]_{\left |\boldsymbol{\nu }\right |} }{\boldsymbol{R}^{\boldsymbol{\nu }}}(
\left |\boldsymbol{\nu }\right |!)^{\delta -1}.
\end{align}
The regularity of $\Delta \partial ^{\boldsymbol{\nu }} u$ can now be established
by bounding the terms in the right-hand side of \reftext{\eqref{vnu-Delta-u}}. Indeed,
for $\boldsymbol{\nu }=\boldsymbol{0}$, we use \reftext{\eqref{Delta_u_variational_form_VTEX1}} and derive the estimate
%
\begin{align}
\label{C_delta}
C_{m}^{2}\|\Delta u\|_{L^{2}(D)} &\leq \|b u^{m} - C_{m}^{2}\nabla a
\cdot \nabla u - C_{m}\,f\|_{L^{2}(D)} \leq \left \|{b}\right \|_{L^{
		\infty }(D)} \left \|{u^{m}}\right \|_{L^{2}(D)} + C_{m}^{2} \left \|{
	\nabla a}\right \|_{L^{\infty }(D)} \left \|{\nabla u}\right \|_{L^{2}(D)}
+ C_{m} \left \|{f}\right \|_{L^{2}(D)}
\notag
\\
&\leq \frac{1}{2} \left ( \overline{b} (C_{m}^{-1}\overline{u})^{m} + C_{m}
\overline{a}\, \overline{u} + C_{m} \overline{f}\right ).
\end{align}
For all
$\boldsymbol{\nu }\in \mathcal F\setminus \left \{\boldsymbol{0}
\right \}$, the Leibniz general product rule and the triangle inequality
yield for the first term of \reftext{\eqref{vnu-Delta-u}}
%
\begin{equation}
\label{term-1}
\begin{split} \|\partial ^{\boldsymbol{\nu }} &(b u^{m} - C_{m}^{2}
	\nabla a \cdot \nabla u - C_{m}\,f)\|_{L^{2}(D)}
	\\
	&\leq C_{m} \|\partial ^{\boldsymbol{\nu }} f\|_{L^{2}(D)} + \sum _{
		\boldsymbol{0}\leq \boldsymbol{\eta }\leq \boldsymbol{\nu }}
	\binom{\boldsymbol{\nu }}{\boldsymbol{\eta }} \big( \|\partial ^{
		\boldsymbol{\nu }-\boldsymbol{\eta }} b\|_{L^{\infty}(D)} \|\partial ^{
		\boldsymbol{\eta }} (u^{m})\|_{L^{2}(D)} + C_{m}^{2}\|\partial ^{
		\boldsymbol{\nu }-\boldsymbol{\eta }} \nabla a\|_{L^{\infty}(D)} \|
	\partial ^{\boldsymbol{\eta }}\nabla u\|_{L^{2}(D)} \big)
	\\
	&\leq C_{m}\|\partial ^{\boldsymbol{\nu }} f\|_{L^{2}(D)}
	+ \|\partial ^{\boldsymbol{\nu }} b\|_{L^{\infty }(D)}
	\|u^{m}\|_{L^{2}(D)} + C_{m}^{2}\|\partial ^{\boldsymbol{\nu }}
	\nabla a\|_{L^{\infty }(D)} \|\nabla u\|_{L^{2}(D)}
	\\
	&\quad +
	\frac{(|\boldsymbol{\nu }|!)^{\delta -1}}{\boldsymbol{R}^{\boldsymbol{\nu }}}
	\sum _{
		\boldsymbol{0}< \boldsymbol{\eta }\leq
		\boldsymbol{\nu }} \binom{\boldsymbol{\nu }}{\boldsymbol{\eta }}
	\big( 3^{m-1} (C_{2m-1} C_{m}^{-1} C_{u})^{m} \overline{b} + C_{m} C_{u}
	\overline{a} \big) \rho ^{|\boldsymbol{\eta }|-1} \left [\tfrac{1}{2}
	\right ]_{|\boldsymbol{\nu }-\boldsymbol{\eta }|} \left [\tfrac{1}{2}
	\right ]_{|\boldsymbol{\eta }|}
	\\
	&\leq \big(
	\overline{b}\,(C_{m}^{-1}\overline{u})^{m} + C_{m}
	\overline{a}\, \overline{u} + 3^{m}(C_{2m-1} C_{m}^{-1}
	C_{u})^{m} \overline{b} + 3C_{m} C_{u}
	\overline{a} +C_{m}\overline{f} \big)
	\frac{ \rho ^{|\boldsymbol{\nu }|-1} \left [\tfrac{1}{2} \right ]_{|\boldsymbol{\nu }|} }{\boldsymbol{R}^{\boldsymbol{\nu }}}(|
	\boldsymbol{\nu }|!)^{\delta -1}{.}
\end{split}
%
\end{equation}
In the above estimate we used \reftext{\eqref{multiindex-est-3}} and that the
following bounds hold for all
$\boldsymbol{\eta }\in \mathcal F\setminus \left \{\boldsymbol{0}
\right \} $
\begin{equation*}
\|\partial ^{\boldsymbol{\eta }} u\|_{L^{2m}(D)} \leq C_{2m-1} \|
\partial ^{\boldsymbol{\eta }} u\|_{H^{1}_{0}(D)} = C_{2m-1} C_{m}^{-1}
\|\partial ^{\boldsymbol{\eta }} u\|_{V} \leq C_{2m-1} C_{m}^{-1} C_{u}
\frac{ \rho ^{|\boldsymbol{\eta }|-1} \left [\tfrac{1}{2} \right ]_{|\boldsymbol{\eta }|} }{\boldsymbol{R}^{\boldsymbol{\eta }}}(|
\boldsymbol{\eta }|!)^{\delta -1},
\end{equation*}
i.e. the Sobolev embedding for $(d,2m-1) \in \mathcal M$, and \reftext{Theorem~\ref{gevrey_regularity_for_semilinear_VTEX1}}. By \reftext{Lemma~\ref{upper_bound_for_f_derivative_VTEX1}} this implies
\begin{equation*}
\|\partial ^{\boldsymbol{\eta }} (u^{m})\|_{L^{2}(D)} = \|\partial ^{
	\boldsymbol{\eta }} (u^{m})\|_{\mathcal L(2m/m)} \leq
3^{m-1} (C_{2m-1} C_{m}^{-1} C_{u})^{m}\,
\frac{ \rho ^{|\boldsymbol{\eta }|-1} \left [\tfrac{1}{2} \right ]_{|\boldsymbol{\eta }|} }{\boldsymbol{R}^{\boldsymbol{\eta }}}(|
\boldsymbol{\eta }|!)^{\delta -1}
\end{equation*}
and hence \reftext{\eqref{term-1}}. The argument can now be completed by induction,
which is the subject of the following Theorem.
%
\begin{theorem}%
\label{thm:2}
Suppose the assumptions of \reftext{Theorem~\ref{gevrey_regularity_for_semilinear_VTEX1}} hold and additionally \reftext{Assumption~\ref{Assumption1}} is satisfied. Then
$\Delta \partial ^{\boldsymbol{\nu }} u \in L^{2}(D)$ for all
$\boldsymbol{\nu }\in \mathcal F\setminus \left \{
\boldsymbol{0}\right \}$ and there holds
%
\begin{equation}
	\label{Delta-u-derivative}
	C_{m}^{2}\|\partial ^{\boldsymbol{\nu }}\Delta u\|_{L^{2}(D)} \leq
	\frac{2 {C}_{\Delta} \tilde{\rho}^{|\boldsymbol{\nu }|-1} \left [\tfrac{1}{2} \right ]_{|\boldsymbol{\nu }|} }{\boldsymbol{R}^{\boldsymbol{\nu }}}
	(|\boldsymbol{\nu }|!)^{\delta -1},
\end{equation}
where $\tilde \rho := \max (4\overline{a},\rho )$ and
\begin{align*}
	C_{\Delta } = \left (\left (1+
	\frac{\overline{a}}{2}\right )\left (C_{m}\overline{f} + \overline{b}
	\,(C_{m}^{-1}\overline{u})^{m} + C_{m} \overline{a}\, \overline{u}
	\right ) + (3C_{2m-1} C_{m}^{-1} C_{u})^{m} \overline{b} + 3C_{m} C_{u}
	\overline{a} \right ).
\end{align*}
\end{theorem}
\begin{proof}
We prove the statement by induction. For
$\boldsymbol{\nu }= \boldsymbol{e}$ we have by \reftext{\eqref{C_delta}} and \reftext{\eqref{term-1}}
\begin{equation*}
	\begin{split} C_{m}^{2}\|\partial ^{\boldsymbol{e}}\Delta u\|_{L^{2}(D)}
		&\leq \| \partial ^{\boldsymbol{e}}( bu^{m} - C_{m}^{2}\nabla a
		\cdot \nabla u - C_{m} f)\|_{L^{2}(D)} + \left \|{\partial ^{
				\boldsymbol{e}}\nabla a}\right \|_{L^{\infty }(D)}C_{m}^{2} \left \|{
			\Delta u}\right \|_{L^{2}(D)}
		\\
		&\leq \big ( \overline{b}\,(C_{m}^{-1}\overline{u})^{m} + C_{m}
		\overline{a}\, \overline{u} + (3C_{2m-1} C_{m}^{-1} C_{u})^{m}
		\overline{b} + 3C_{m} C_{u} \overline{a} +C_{m}\overline{f} \big )
		\frac{ \left [\tfrac{1}{2} \right ]_{1} }{\boldsymbol{R}^{\boldsymbol{e}}}
		+ \frac{1}{2}\big ( C_{m}\overline{f} + \overline{b}\,(C_{m}^{-1}
		\overline{u})^{m} + C_{m} \overline{a}\, \overline{u}\big )
		\frac{ \overline{a}\left [\tfrac{1}{2} \right ]_{1} }{\boldsymbol{R}^{\boldsymbol{e}}}\\
		&
		= {C}_{\Delta }
		\frac{ \left [\tfrac{1}{2} \right ]_{1} }{\boldsymbol{R}^{\boldsymbol{e}}}
		\leq 2 {C}_{\Delta }
		\frac{ \left [\tfrac{1}{2} \right ]_{1} }{\boldsymbol{R}^{\boldsymbol{e}}}.
	\end{split}
\end{equation*}
The aim is now to prove \reftext{\eqref{Delta-u-derivative}}, assuming that it is
valid for all derivatives of order strictly less than
$\boldsymbol{\nu }$. From \reftext{\eqref{vnu-Delta-u}} and the bound below it, we
readily have for $|\boldsymbol{\nu }|\geq 2$
\begin{equation*}
	\begin{split} C_{m}^{2} \|\partial ^{\boldsymbol{\nu }} \Delta u\|_{L^{2}(D)}
		&
		\leq \|\partial ^{\boldsymbol{\nu }} (b u^{m} - C_{m}^{2}
		\nabla a \cdot \nabla u - C_{m}\,f)\|_{L^{2}(D)} + \|\partial ^{
			\boldsymbol{\nu }}\nabla a\|_{L^{\infty }(D)} C_{m}^{2} \|\Delta u\|_{L^{2}(D)} \\
		& \qquad + \sum _{\boldsymbol{0}< \boldsymbol{\eta }<
			\boldsymbol{\nu }} \binom{\boldsymbol{\nu }}{\boldsymbol{\eta }} \|
		\partial ^{\boldsymbol{\nu }-\boldsymbol{\eta }}\nabla a\|_{L^{
				\infty }(D)} C_{m}^{2} \|\partial ^{\boldsymbol{\eta }}\Delta u\|_{L^{2}(D)}
		\\
		&\leq C_{\Delta }
		\frac{ \rho ^{|\boldsymbol{\nu }|-1} \left [\tfrac{1}{2} \right ]_{|\boldsymbol{\nu }|} }{\boldsymbol{R}^{\boldsymbol{\nu }}}(|
		\boldsymbol{\nu }|!)^{\delta -1} + \sum _{
			\boldsymbol{0}< \boldsymbol{\eta }<
			\boldsymbol{\nu } } \binom{\boldsymbol{\nu }}{\boldsymbol{\eta }}
		\frac{\overline{a}\left [\tfrac{1}{2} \right ]_{|\boldsymbol{\nu }-\boldsymbol{\eta }|} }{\boldsymbol{R}^{\boldsymbol{\nu }-\boldsymbol{\eta }}}
		(|\boldsymbol{\nu }-\boldsymbol{\eta }|!)^{\delta -1}
		\frac{2{C}_{\Delta }\tilde{\rho}^{|\boldsymbol{\eta }|-1}\left [\tfrac{1}{2} \right ]_{|\boldsymbol{\eta }|} }{\boldsymbol{R}^{\boldsymbol{\eta }}}
		(|\boldsymbol{\eta }|!)^{\delta -1}
		\\
		&\leq 2 {C}_{\Delta }
		\frac{ \tilde \rho ^{|\boldsymbol{\nu }|-1} \left [\tfrac{1}{2} \right ]_{|\boldsymbol{\nu }|} }{\boldsymbol{R}^{\boldsymbol{\nu }}}(|
		\boldsymbol{\nu }|!)^{\delta -1} \big(\frac{1}{2} +
		\frac{2\overline{a}}{\tilde \rho}\big).
	\end{split}
\end{equation*}
According to the definition of $\tilde \rho $ the term in the parentheses
is bounded by 1 and the claim of the theorem follows.
\end{proof}

Finally we remark that
$(\|u\|_{L^{2}(D)}^{2} + \|\Delta u\|_{L^{2}(D)}^{2})^{1/2}$ is an equivalent
norm on $H^{2}(D) \cap H^{1}_{0}(D)$. Moreover, if $D$ is convex, the solution
$u \in H^{1}_{0}(D)$ to \reftext{\eqref{variational_form_VTEX1}} is in $H^{2}(D)$ by the
linear elliptic regularity theory, see e.g. \cite{Grisvard1985}. Indeed,
the solution $u\in H^{1}_{0}(D)$ to \reftext{\eqref{variational_form_VTEX1}} can be seen
as the unique solution of
\begin{equation*}
\int _{D} a \nabla w \cdot \nabla v = \int _{D} \tilde f v, \qquad
\forall v \in H^{1}_{0}(D),
\end{equation*}
where $\tilde f := C_{m}^{-1} f - C_{m}^{-2} b u^{m}$. Since
$H^{1}_{0}(D)$ is continuously embedded in $L^{2m}(D)$ and
$\tilde f \in L^{2}(D)$, the linear elliptic regularity theory in convex
domains applies. Recalling \reftext{\eqref{ff-estimates}},
for all
$\boldsymbol{\nu }\in \mathcal F\setminus \left \{\boldsymbol{0}
\right \}$ we have as a result in this case
\begin{equation*}
\|\partial ^{\boldsymbol{\nu }}u\|_{H^{2}(D)} \lesssim (\|\partial ^{
	\boldsymbol{\nu }}u\|_{L^{2}(D)}^{2} + \|\Delta \partial ^{
	\boldsymbol{\nu }} u\|_{L^{2}(D)}^{2})^{1/2} \leq ((C_{1}C_{m}^{-1}C_{u})^{2}
+ (2 C_{m}^{-2} C_{\Delta })^{2})^{1/2}
\frac{ (|\boldsymbol{\nu }|!)^{\delta}}{(\boldsymbol{R}/\tilde \rho )^{\boldsymbol{\nu }}},
\end{equation*}
where $C_{1}$ is the constant of the Sobolev embedding
$H^{1}_{0}(D)\hookrightarrow L^{2}(D)$, i.e.
$\left \|{v}\right \|_{L^{2}(D)}\leq C_{1} \left \|{ \nabla v}\right
\|_{L^{2}(D)}$.

\section{Applications and numerical experiments}
\label{sec:_num_exp_VTEX1}

In this section we give two numerical examples that demonstrate how the
abstract regularity results of \reftext{Theorem~\ref{gevrey_regularity_for_semilinear_VTEX1}} and \reftext{Theorem~\ref{thm:2}} can be applied
to mathematically predict convergence of numerical methods for nonlinear
reaction-diffusion problems under uncertainty.

\subsection{Gauss-Legendre quadrature}
\label{sec:_Gauss-Legendre_Quadrature_VTEX1}

According to \reftext{Theorem~\ref{gevrey_regularity_for_semilinear_VTEX1}}, under appropriate
assumptions the solution of the semilinear problem \reftext{\eqref{variational_form_VTEX1}} inherits the Gevrey-$\delta $ class regularity
from the coefficients and the right-hand side of this equation. To demonstrate
this result in a numerical experiment, we first consider a simple model
problem with a single scalar parameter $y$ being a random variable uniformly
distributed in $[-1,1]$. The quantity of interest is defined as the integral
%
\begin{equation}
\label{Elambda}
I({\mathcal{G}}) := \int _{-1}^{1} {\mathcal{G}}(y) \, dy,
\end{equation}
where the functional
${\mathcal{G}}(y) := \int _{D} u(\boldsymbol{x}, y) \, d
\boldsymbol{x}$ is the average value of the solution $u$ in the computational
domain $D= (0,1)^{2}$. Specifically, we consider the cubic nonlinearity
$m=3$,
the forcing term
$f= 3(\cos (2 \pi x_{1})+1)(\cos (3\pi x_{2})+1)$, the unit diffusion
coefficient $a\equiv 1$ and the reaction coefficient $b$ being either
%
\begin{equation}
\label{def-b1}
b^{(1)} (\boldsymbol{x},y) = 50(\cos ^{2}(15\pi x_{1}+y^{10})+1)(
\cos ^{2}(17\pi x_{2}+y^{25})+1),
\end{equation}
or
%
\begin{equation}
\label{def-b2}
b^{(2)} (\boldsymbol{x},y) = \left (\exp \left (-
\frac{x_{1}^{2}+x_{2}^{2}}{y+1}\right )+1\right ) (\cos ^{2}(15\pi x_{1})+1)(
\cos ^{2}(17\pi x_{2})+1).
\end{equation}
Since $m$ is odd and both $b^{(1)}$ and $b^{(2)}$ are nonnegative, \reftext{Assumption~\ref{Possitive_b_Assump_VTEX1}} is valid and therefore corresponding solutions
$u^{(1)}$ and $u^{(2)}$ of \reftext{\eqref{variational_form_VTEX1}} are uniquely determined
in $H^{1}_{0}(D)$ for every~$y\in [-1,1]$. Recall that
$H^{1}_{0}(D)$ is continuously embedded in $L^{1}(D)$ and therefore, by
\reftext{Theorem~\ref{gevrey_regularity_for_semilinear_VTEX1}}
%
\begin{equation}
|\partial ^{\boldsymbol{\nu }}{\mathcal{G}}| \leq \|\partial ^{
	\boldsymbol{\nu }} u \|_{L^{1}(D)} \leq C_{0} \|\partial ^{
	\boldsymbol{\nu }} u \|_{H^{1}_{0}(D)} \leq C_{0}C_{m}^{-1}C_{u}
\frac{(|\boldsymbol{\nu }|!)^{\delta}}{(\boldsymbol{R}/\rho )^{\boldsymbol{\nu }}},
\label{eq59}
\end{equation}
i.e.
${\mathcal{G}}={\mathcal{G}}^{(k)}$ is Gevrey-$
\delta $ regular, where $\delta =\delta ^{(k)}$ is determined by
$b^{(k)}$, $k = 1,2$. More precisely, we have
\begin{align*}
b^{(1)}&
\in G^{1}([-1,1],L^{\infty }(D)),\quad {\mathcal{G}}^{(1)}
\in G^{1}([-1,1],\mathbb R),\,
\\
b^{(2)}&
\in G^{2}([-1,1],L^{\infty }(D)),\quad {\mathcal{G}}^{(2)}
\in G^{2}([-1,1],\mathbb R).
\end{align*}
We denote by $\{\xi _{j},w_{i}\}_{j=1}^{n}$ the Gauss-Legendre quadrature
points and weights and hereby the quadrature approximation to the quantity
of interest~\reftext{\eqref{Elambda}}
%
\begin{equation}
Q_{n}[{\mathcal{G}}] := \sum _{j=1}^{n} w_{j}
{\mathcal{G}}(\xi _{i}).
\label{eq60}
\end{equation}
According to \cite[Theorem 5.2]{ChernovSchwab2012} the asymptotic quadrature
error is determined by the Gevrey class parameter of the data, that is
$\delta ^{(1)} = 1$ for \reftext{\eqref{def-b1}} and $\delta ^{(2)} = 2$ for \reftext{\eqref{def-b2}}: $\exists C,r > 0$ independent on $n$ such that
%
\begin{equation}
\label{eps2}
\varepsilon _{n}^{(k)} := |I({\mathcal{G}}^{(k)}) - Q_{n}[{\mathcal{G}}^{(k)}]|
\leq C \exp ( - r n^{1/\delta ^{(k)}}), \qquad k=1,2.
\end{equation}
Observe in particular that unlike $b^{(1)}$, the coefficient
$b^{(2)}$ is not analytic near $y=-1$, see e.g.
\cite[Section 5.2]{ChernovSchwab2012}), and this can be immediately detected
by the reduced asymptotic convergence rate for the case~$k=2$.

At every quadrature point $y = \xi _{i}$, we solve the variational problem \reftext{\eqref{variational_form_VTEX1}} with the finite element method on a uniform triangular
mesh with the mesh size $h=\sqrt{2}/128$. Since \reftext{\eqref{variational_form_VTEX1}} is a nonlinear problem, we use the fixed-point
iteration method \reftext{\eqref{contracting_mapping_VTEX1}} with an absolute error tolerance
of $10^{-14}$ with respect to the $H^{1}_{0}(D)$-seminorm. Since
$I({\mathcal{G}}(u))$ is not available in closed form, we replace it by very
fine quadrature approximations:
$I^{*}({\mathcal{G}}(u^{(1)}))=Q_{50}[{\mathcal{G}}(u^{(1)})]$ and
$I^{*}({\mathcal{G}}(u^{(2)}))=Q_{150}[{\mathcal{G}}(u^{(2)})]$ based on
$n=50$ and $n=150$ Gauss-Legendre points respectively.

\begin{figure*}
\includegraphics[width=\textwidth]{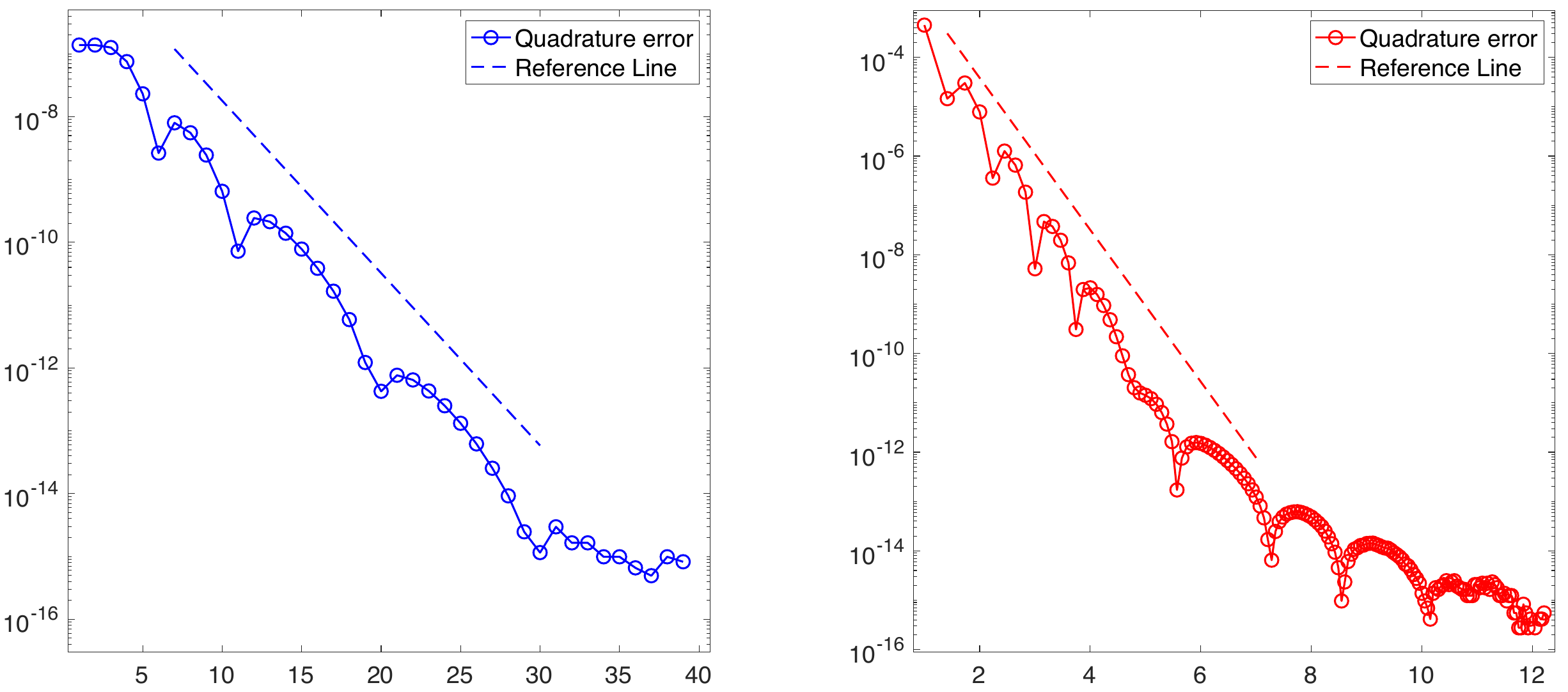}
\caption{Quadrature error $\varepsilon ^{(1)}_{n}$ (left) with respect to the
	number $n$ of quadrature points and Quadrature error $\varepsilon ^{(2)}_{n}$
	(right) with respect to $N = n^{1/2}$.}
\label{fig1}
\end{figure*}

In the left panel of \reftext{Fig.~\ref{fig1}}, the data points represent the quadrature
error $\varepsilon _{n}^{(1)}$ with respect to number of quadrature points
$n$ and the reference line (dashed line) shows the linear trend in semi-logarithmic
scale. This agrees with \reftext{\eqref{eps2}} and $\delta ^{(1)}=1$. Similarly,
the data points in the right panel of \reftext{Fig.~\ref{fig1}} represent the quadrature
error $\varepsilon _{n}^{(2)}$ with respect to the square root of the number
of quadrature points $N :=n^{1/ \delta ^{(2)}} = n^{1/2}$ in the semi-logarithmic
scale. The reference line (dashed line) also shows the linear trend in
semi-logarithmic scale. Clearly, this confirms \reftext{\eqref{eps2}} with
$\delta ^{(2)}=2$ and thereby demonstrates the meaning and validity of
\reftext{Theorem~\ref{gevrey_regularity_for_semilinear_VTEX1}}.

\subsection{Quasi-Monte Carlo method for Gevrey functions}
\label{sec:_QMC_VTEX1}

In this section we give an example for the case of high-dimensional parametric
integration and demonstrate an application of \reftext{Theorem~\ref{thm:2}}. For
this, we consider the variational problem \reftext{\eqref{variational_form_VTEX1}} with
$s$ parameters $\boldsymbol{y}= (y_{1},\dots ,y_{s})$ and the quantity
of interest being the integral
%
\begin{align}
\label{approx_integral_VTEX1}
I_{s}(F) = \int _{U} F(
\boldsymbol{y}) \, d y_{1} \dots dy_{s},
\end{align}
where the functional
$F(\boldsymbol{y}):=u(\boldsymbol{x}_{0},
\boldsymbol{y})$ is the point-evaluation of the solution $u$ in the center
of the computational domain $D= (0,1)^{2}$ and the $s$-dimensional unit
cube $U = [-\frac{1}{2}, \frac{1}{2}]^{s}$ is the parameter domain. Such
problems arise when the coefficients of \reftext{\eqref{variational_form_VTEX1}} are modelled
as general random fields e.g. via their Karhunen-L\`oeve expansion
\cite{HansenSchwab13,CohenDevoreSchwab2011,KuoSchwabSloan2013,KuoSchwabSloan2012}
and then truncated to a certain dimension $s$ for computational reasons.
For further comments on the analysis of the truncation error for Gevrey-class
parametrizations we refer to \cite[Section 6.2]{ChernovLe2023} and further
references therein.

In this section we fix the dimension of the parameter domain at
$s = 100$ and concentrate on estimation of $I_{s}(u)$ where $u$ is a solution
for the variational formulation \reftext{\eqref{variational_form_VTEX1}} with
$a\equiv 1$, $f\equiv 1$, the reaction coefficient $b$ being one of following
functions
%
\begin{equation}
\label{QMC-a1-new}
b^{(1)} (\boldsymbol{x},\boldsymbol{y}) = 2+2\exp \left (-{\zeta (5)}+
\sum _{j=1}^{s} {j^{-5}} \sin (j\pi x_{1})\sin (j\pi x_{2})\,y_{j}
\right ),
\end{equation}
or
%
\begin{equation}
\label{QMC-a2-new}
b^{(2)} (\boldsymbol{x},\boldsymbol{y}) = 3+\frac{1}{\zeta (5)} \sum _{j=1}^{s}
j^{-5} \sin (j\pi x_{1})\sin (j\pi x_{2})\exp \left ( -
\frac{1}{y_{j}+\tfrac{1}{2}}\right ).
\end{equation}
Here $y_{j}$ are independent uniformly distributed in
$[-\frac{1}{2},\frac{1}{2}]$ random variables for all
$1 \leq j\leq s$. Notice that $b^{(1)}$ is analytic with respect to
$\boldsymbol{y}$ ($\delta ^{(1)} = 1$), whereas $b^{(2)}$ is of Gevrey-$
\delta $ class with $\delta ^{(2)} = 2$,
i.e. $ b^{(1)} \in G^{1}(U,L^{\infty }(D))$ and
$ b^{(2)} \in G^{2}(U,L^{\infty }(D))$. Notice also that both
$b^{(1)}$ and $b^{(2)}$ are $\ell ^{p}$-summable with $p$ less than one.
More precisely, with any $p > \frac{1}{5}$. This will be important for
the convergence of the Quasi-Monte-Carlo (QMC) methods,
as explained below.

\reftext{Theorem~\ref{gevrey_regularity_for_semilinear_VTEX1}} and \reftext{Theorem~\ref{thm:2}} imply that $F(\boldsymbol{y})$ is Gevrey-$
\delta $ regular with the same $\delta $ as for $b^{(1)}$ and
$b^{(2)}$. This follows since $H^{2}(D)$ is continuously embedded in
$L^{\infty}(D)$ and
$(\|u\|_{L^{2}(D)}^{2} + \|\Delta u\|_{L^{2}(D)}^{2} )^{1/2}$ is an equivalent
norm on $H^{2}(D)$.

The quantity of interest for the coefficient $b^{(k)}$ will be denoted
by $I_{s}(F^{(k)})$, $k = 1,2$. The standard Monte
Carlo (MC) quadrature is the sample mean approximation
$Q^{MC}_{s,n}(F) := \frac{1}{n} \sum _{i=1}^{n}
F(\boldsymbol{y}^{(i)})$, where
$\boldsymbol{y}^{(i)}$ are independent samples uniformly distributed
in $U$. It satisfies the classical error estimate for the relative root
mean square error
%
\begin{equation}
\label{errorMC}
\varepsilon ^{\textrm{MC},(k)}_{n} :=
\sqrt{\mathbb E\left ( \left |\frac{I_{s}(F^{(k)})-Q^{MC}_{s,n}(F^{(k)})}{I_{s}(F^{(k)})}\right |^{2}\right )}
\lesssim C n^{-\frac{1}{2}}
\end{equation}
for a positive constant $C$ independent of $n$. Observe that the asymptotic
error bound is insensitive to the Gevrey class regularity of
$F$. As an alternative to the standard
(MC) quadrature, we consider a class of
QMC rules called randomly shifted rank-1 lattice rules
defined as
%
\begin{align}
\label{QMC_quad_def_VTEX1}
Q^{\Delta}_{s,n}(F) := \frac{1}{n} \sum _{i=1}^{n}
F\big(\left \{ \tfrac{i \boldsymbol{z}_{s}}{n} +
\Delta \right \} -\tfrac{1}{2}\big).
\end{align}
Here the braces in \reftext{\eqref{QMC_quad_def_VTEX1}} indicate the fractional part of
each component of the argument vector, see for e.g.
\cite{KuoNuyens2016}, \cite{KuoSchwabSloan2012}. The quadrature points
are constructed using a generating vector
$\boldsymbol{z}_{s} \in \mathbb N^{s}$, and a random shift $\Delta $, which
is a random variable uniformly distributed over the cube $(0,1)^{s}$. Since
$Q^{\Delta}_{s,n}(G)$ is also a random variable, we use the relative root
mean square error (RMSE) as a measure of accuracy
%
\begin{align}
\label{error-deco}
\varepsilon ^{QMC,(k)}_{n} :=
\sqrt{\mathbb{E}\left (\bigg|\frac{I_{s}(F^{(k)})-Q^{\Delta}_{s,n}(F^{(k)})}{I_{s}(F^{(k)})}\bigg|^{2}\right )},
\end{align}
where $\mathbb E$ stands for the expectation with respect to the random
shifts $\Delta $. The analysis for RMSE, when the integrand is analytic
(of Gevrey class with $\delta =1$) can be found in
\cite{KuoNuyens2016} and \cite{KuoSchwabSloan2012}. In
\cite[Lemma 7.4]{ChernovLe2023} we extended this result to integrands that
may belong to Gevrey-$\delta $ class with $\delta > 1$. According to that,
the RMSE of the QMC quadrature for a fixed integer $s$ and $n$ being a power
of 2 admits the bound
%
\begin{equation}
\label{errorQMC}
\varepsilon ^{QMC,(k)}_{n} \leq C_{s,\vartheta} n^{-
	\frac{1}{2\vartheta}},
\end{equation}
where $C_{s,\vartheta}$ is independent of $n$ and
\begin{align*}
\vartheta = \left \{
\begin{matrix}
	\omega &\text{for any } \omega \in (\tfrac{1}{2},1) & \text{when } p
	\in (0,\tfrac{2}{3\delta}],
	\\
	\frac{\delta p}{2-\delta p} & & \text{when } p\in (\tfrac{2}{3\delta},
	\tfrac{1}{\delta}] .
\end{matrix}\right .
\end{align*}

Recall that the summability parameter $p$ can be chosen to satisfy
$\frac{1}{5}< p <\frac{2}{3 \delta ^{(k)}}$, $k= 1,2$ and therefore
$\vartheta $ can be close to~$\frac{1}{2}$ and thereby
$\varepsilon _{n}^{{\mathrm{MC}},(k)}$ is close to $n^{-1}$. This effect is
very well visible in the numerical experiments in \reftext{Fig.~\ref{fig2}}. These
results were obtained by solving the variational problem \reftext{\eqref{variational_form_VTEX1}} in each quadrature point by the finite element
method on a uniform mesh with the mesh size $h=\sqrt{2}/128$ and by means
of a fixed-point iteration \reftext{\eqref{contracting_mapping_VTEX1}} with the error tolerance
of $10^{-14}$ with respect to the $H^{1}_{0}(D)$-seminorm.

The outer expectation in \reftext{\eqref{errorQMC}} is approximated by the empirical
mean of $R = 8$ runs, i.e., for $\Delta ^{(j)}$, $1 \leq j \leq R$ being
an independent sample from the uniform distribution from the unit cube
$(0,1)^{s}$ and $Q^{(j)}_{s,n}(F^{(k)})$ the corresponding
QMC quadrature, we approximate the relative QMC error by
%
\begin{equation}
\varepsilon _{n}^{{\mathrm{QMC}},(k)}\sim
\sqrt{
	\frac{1}{R} \sum _{j=1}^{R}
	\left (\left |\frac{ I_{s}^{*}(F^{(k)})-Q^{(j)}_{s,n}(F^{(k)})}{I_{s}^{*}(F^{(k)})}\right |^{2}\right )}
\label{eq69}
\end{equation}
and analogously for the standard MC approximation
$\varepsilon _{n}^{{\mathrm{MC}},(k)}$. We choose the reference value
$I_{s}^{*}(F^{(k)})$ as QMC approximation with a very
high number of points, namely $n=2^{15}$, in both cases $k=1,2$.
It is worth noting that the relative error
$\varepsilon _{n}^{{\mathrm{QMC}},(k)}$ is independent of dimension $s$, as
shown in, for example, \cite[Lemma 7.4]{ChernovLe2023}.

\reftext{Fig.~\ref{fig2}} clearly demonstrates that predicted MC and QMC convergence
rates of $n^{-\frac{1}{2}}$ and about $n^{-1}$, cf. \reftext{\eqref{errorMC}} and \reftext{\eqref{errorQMC}}, are very well reproduced in the numerical experiments.

\begin{figure}
	\begin{center}
	\includegraphics[width=0.5\textwidth]{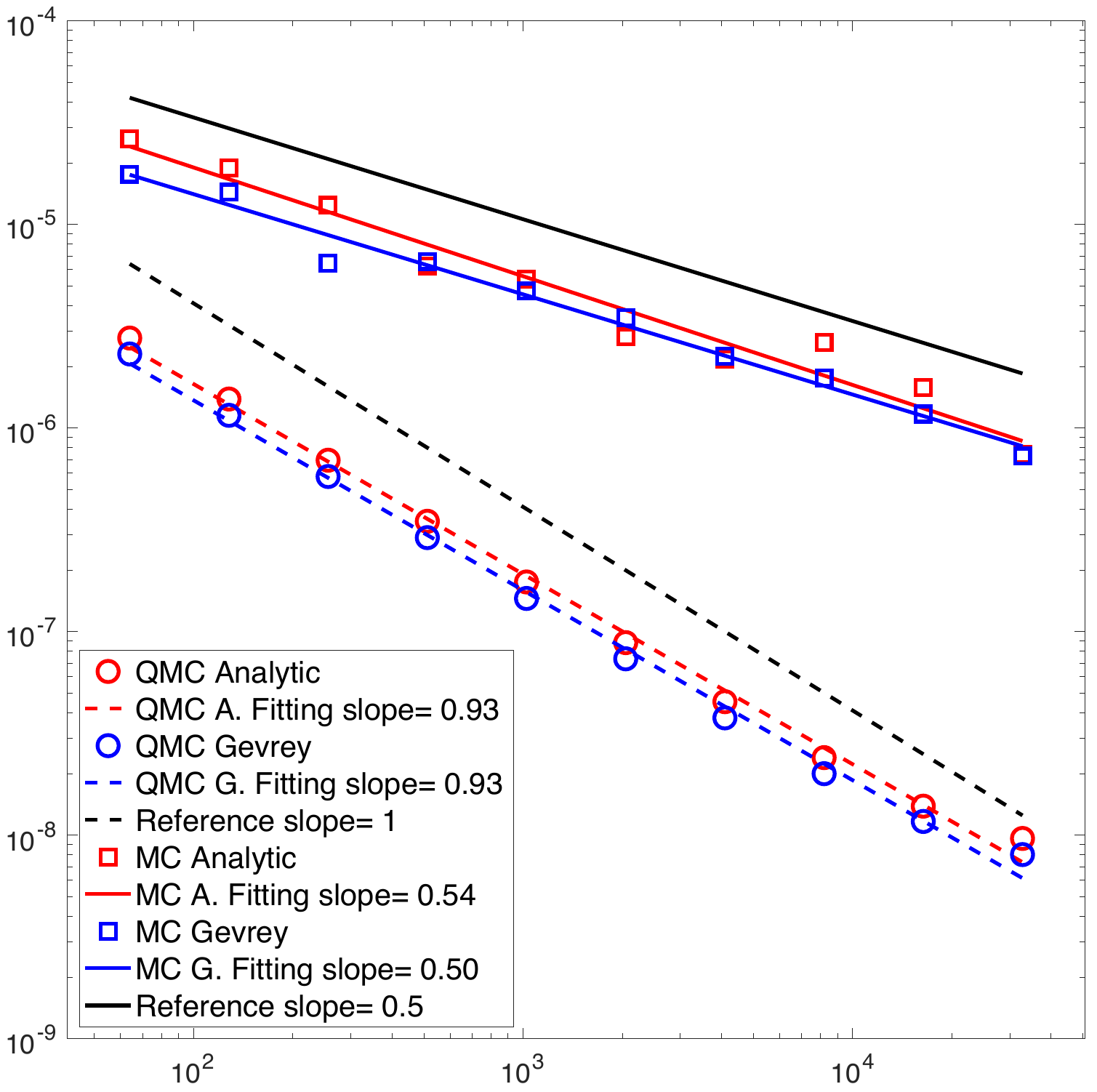}
	\end{center}
\caption{Convergence of the quadrature error with respect to the number of
	samples $n$ for four methods: QMC analytic ($\varepsilon
	^{{\mathrm{QMC}},(1)}_{n}$), QMC Gevrey ($\varepsilon
	^{{\mathrm{QMC}},(2)}_{n}$), MC analytic ($ \varepsilon
	^{{\mathrm{MC}},(1)}_{n}$), MC Gevrey ($\varepsilon ^{{\mathrm{MC}},(2)}_{n}$).}%
\label{fig2}
\end{figure}




\bibliographystyle{elsarticle-num-names} 




\end{document}